\documentclass[11pt]{amsart}
\usepackage{txfonts}
\usepackage{amscd,amssymb,amsmath,amsbsy,amsthm}
\usepackage[colorlinks,plainpages,backref,urlcolor=blue,breaklinks]{hyperref}
\usepackage{tikz}
\usetikzlibrary{cd,calc,intersections}
\usetikzlibrary{svg.path}
\usepackage{enumitem}
\usepackage{nccmath, relsize}
\usepackage{xfrac}
\usepackage{microtype}
\usepackage{verbatim}
\usepackage[all]{xy}
\usepackage{bm}
\usepackage{tikz-cd} 
\usepackage{graphicx} 

\usepackage{caption}
\usepackage{subcaption} 
\usepackage{colortbl}
\usepackage{color}

\newtheorem{theorem}{Theorem}[section]
\newtheorem{lemma}[theorem]{Lemma}
\newtheorem{corollary}[theorem]{Corollary}
\newtheorem{proposition}[theorem]{Proposition}

\theoremstyle{definition}
\newtheorem{definition}[theorem]{Definition}
\newtheorem{example}[theorem]{Example}
\newtheorem{remark}[theorem]{Remark}
\newtheorem{question}[theorem]{Question}

\newtheorem{theoremalph}{Theorem}

\newcommand{\N}{\mathbb{N}}
\newcommand{\Z}{\mathbb{Z}}
\newcommand{\Q}{\mathbb{Q}}
\newcommand{\R}{\mathbb{R}}
\newcommand{\C}{\mathbb{C}}
\newcommand{\F}{\mathbb{F}}
\newcommand{\K}{\mathbb{K}}

\newcommand{\x}{\mathbf{x}}

\newcommand{\bo}{\mathbf{1}}

\renewcommand{\k}{\Bbbk}

\newcommand{\VV}{\mathcal{V}}

\newcommand{\PP}{\mathcal{P}}

\newcommand{\cJ}{\mathcal{J}}
\newcommand{\cO}{\mathcal{O}}

\newcommand{\rS}{\mathrm{S}}

\DeclareMathOperator{\rank}{rank}

\DeclareMathOperator{\im}{im}
\DeclareMathOperator{\coker}{coker}
\DeclareMathOperator{\codim}{codim}

\DeclareMathOperator{\id}{id}
\DeclareMathOperator{\ab}{{ab}}

\DeclareMathOperator{\ch}{char}

\DeclareMathOperator{\GL}{GL}
\DeclareMathOperator{\Hom}{{Hom}}
\DeclareMathOperator{\Tor}{{Tor}}

\DeclareMathOperator{\inter}{{int}}

\DeclareMathOperator{\init}{in}

\DeclareMathOperator{\Tors}{Tors}

\DeclareMathOperator{\Trop}{Trop}

\DeclareMathOperator{\orb}{orb}
\DeclareMathOperator{\lcm}{{lcm}}

\DeclareMathOperator{\Fitt}{{Fitt}}

\DeclareMathOperator{\Cay}{Cay}

\DeclareMathOperator{\reg}{{reg}}

\newcommand{\surj}{\twoheadrightarrow}
\newcommand{\inj}{\hookrightarrow}
\newcommand\isom{\xrightarrow{ \,\smash{\raisebox{-0.6ex}{\ensuremath{\scriptstyle\simeq}}}\,}}

\newcommand{\longsurj}{\relbar\joinrel\twoheadrightarrow}
\newcommand{\longinj}{\lhook\joinrel\longrightarrow} 
\newcommand{\compl}{\mathrm{c}}

\newcommand{\wX}{\widetilde{X}}
\newcommand{\abs}[1]{\left| #1 \right|}

\definecolor{lime}{HTML}{A6CE39}
\definecolor{lime}{HTML}{A6CE39}
\DeclareRobustCommand{\orcidicon}{
	\begin{tikzpicture}
	\draw[lime, fill=lime] (0,0) 
	circle [radius=0.16] 
	node[white] {{\fontfamily{qag}\selectfont \tiny ID}};
	\draw[white, fill=white] (-0.0625,0.095) 
	circle [radius=0.007];
	\end{tikzpicture}
	\hspace{-2mm}
}
\foreach \x in {A, ..., Z}{\expandafter\xdef\csname orcid\x\endcsname{\noexpand\href{%
https://orcid.org/\csname orcidauthor\x\endcsname}{\noexpand\orcidicon}}}

\title[Twisted homology jump loci and $\Sigma$-invariants]%
{Twisted homology jump loci, twisted Alexander polynomials, and $\Sigma$-invariants}

\author[Yongqiang Liu]{Yongqiang Liu\,$^1$}
\address{Institute of Geometry and Physics, University of Science 
and Technology of China, 96 Jinzhai Road, Hefei Anhui 230026 China} 
\email{liuyq@ustc.edu.cn}
\thanks{$^1$Partially supported by  NSFC grant No. 12571047, the Project of 
Stable Support for Youth Team in Basic Research Field, CAS (YSBR-001) and the 
starting grant from University of Science and Technology of China.}

\author[Alexander~I.~Suciu]{Alexander~I.~Suciu\,$^2$\!\orcidA{}}
\address{Department of Mathematics, Northeastern University, Boston, MA 02115, USA}
\email{\href{mailto:a.suciu@northeastern.edu}{a.suciu@northeastern.edu}}
\urladdr{\href{http://web.northeastern.edu/suciu/}{web.northeastern.edu/suciu/}}
\thanks{$^2$Partially supported by the project ``Singularities and 
Applications" - CF 132/31.07.2023 funded by the European Union - 
NextGenerationEU - through Romania’s National Recovery and Resilience Plan.}



\date{August 7, 2026}

\begin{document}

\begin{abstract}
We introduce the twisted homology jump loci of a space $X$: the jump
loci for homology with coefficients in rank-one local systems, twisted
by a fixed finite-dimensional representation $\sigma$ of $\pi_1(X)$.
These loci refine the classical characteristic varieties, and their
defining equations in degree one are the twisted Alexander polynomials
of knot theory. Our main theorem is that their tropicalizations bound
from above the Bieri--Neumann--Strebel--Renz (BNSR) $\Sigma$-invariants
of $X$.

Twisting gains real ground. The resulting bound is strictly stronger
than the untwisted tropical bounds obtained from the usual
characteristic varieties: for a one-relator group whose $\Sigma^1$ was
computed by Brown, the untwisted bound excludes only two directions in
$H^1(G;\R)$, whereas the twisted bound determines $\Sigma^1(G)$ exactly.
For a compact orientable $3$-manifold $M$ with toroidal or empty
boundary, the twisted bound is sharp: the union of the twisted tropical
varieties over all finite-image integral representations of $\pi_1(M)$
computes $\Sigma^1(\pi_1(M))$, and hence recovers the fibered faces of
the Thurston norm ball. Sharpness genuinely requires twisting: a
non-fibered class enters the tropical variety through the vanishing of a
twisted Alexander polynomial along it, and the untwisted polynomial need
not vanish. For a compact K\"{a}hler manifold $X$, we prove that the first
twisted Alexander polynomial $\Delta^{\sigma}(X)$ is either $0$ or $1$,
for every representation $\sigma$ over every field, and that
$\Sigma^1(\pi_1(X))$ is controlled by the hyperbolic orbifold
fibrations of $X$ for every $\sigma$. The obstruction to
K\"{a}hlerianity that comes out of this is strictly finer than its
untwisted counterpart: we exhibit groups $G$ with $\Delta(G)=1$ but
$\Delta^{\sigma}(G)\ne 1$, which the twisted test excludes from being
K\"{a}hler and the classical one does not.
\end{abstract}

\maketitle
\tableofcontents

\section{Introduction}
\label{sect:intro}

\subsection{Overview}
\label{intro:overview}
Let $X$ be a connected, finite-type CW-complex with $\pi_1(X)=G$, let
$\k$ be a field, and let $\sigma\colon G\to \GL_r(\k)$ be a
representation, with corresponding local system $L_\sigma$ on $X$.
The rank-one local systems on $X$ are parametrized by the character
group $\Hom(G,\k^{\times})$; tensoring them with $L_\sigma$ embeds this
group into the rank-$r$ representation variety $\Hom(G,\GL_r(\k))$, and
we study the locus along that orbit where homology jumps,
\begin{equation*}
\VV^i(X,L_\sigma)=\big\{\rho \in \Hom(G,\k^{\times}) \mid
H_i(X; L_\sigma\otimes_{\k} L_{\rho}) \ne 0\big\}\, .
\end{equation*}
We call these the \emph{twisted homology jump loci}\/ of $X$.  For
$\sigma$ trivial they are the classical characteristic varieties of $X$;
in degree one, their defining equation is the twisted Alexander
polynomial $\Delta^{\sigma}(X)$ of knot theory.  The twisted jump loci
thus interpolate between two invariants that have been studied
separately, and our purpose here is to show that the interpolation is
useful.

Our main result, Theorem~\ref{thm:bns-twisted-trop-intro}, is that the
tropicalization of $\VV^{\le q}(X,L_\sigma)$---taken with respect to any
rational valuation on $\k$---bounds the BNSR invariant $\Sigma^q(X;\Z)$
from above, subject to a valuation condition on $\sigma$ that holds
automatically whenever $\sigma$ is defined over $\Z$ or has finite
image.  Specializing to $\sigma=\id$ recovers the untwisted tropical
bounds of \cite{Su-mathann} and \cite{LL}.

Three consequences show that the twisting is not a formality.

First, the bound is strictly stronger than what the untwisted theory
gives.  For the one-relator group $G$ of Example~\ref{ex:one-rel},
whose $\Sigma^1$ was computed by Brown \cite{Br}, the untwisted
tropical variety is a single line, and the bound it yields excludes only
two directions on the circle $\rS(G)$; twisting by a $2$-dimensional
representation of $G$ that factors through $S_3$ yields a tropical
variety that determines $\Sigma^1(G)$ exactly.

Second, for $3$-manifolds the twisted bound is not merely stronger but
sharp.  If $M$ is a compact, connected, orientable $3$-manifold with
toroidal or empty boundary, then by Theorem~\ref{thm:sigma-3mfd-intro}
the union of the twisted tropical varieties over all finite-image
integral representations of $\pi_1(M)$ computes $\Sigma^1(\pi_1(M))$
exactly.  By \cite[Thm.~E]{BNS}, this recovers the fibered faces of the
Thurston norm ball.  Twisting is essential to the equality.  A
non-fibered class $\phi$ enters the tropical variety through a
\emph{vanishing}: by \cite{FV13} there is a finite-image $\sigma$ with
$\Delta^{\phi,\sigma}(M)=0$, whence $\VV^1(M,L_\sigma\otimes \C)$
contains the whole subtorus $\im(\phi^*)\cong \C^{\times}$, and so
$\Trop_\C(\VV^1(M,L_\sigma\otimes \C))$ contains the line $\R\cdot\phi$.
Untwisted, the corresponding polynomial need not vanish along a
non-fibered class, and then the tropicalization misses $\phi$ altogether.

Third, running the implication backwards constrains the twisted
invariants of spaces whose BNSR invariants are already known.  For a
compact K\"{a}hler manifold $X$, Delzant's theorem \cite{De10} describes
$\Sigma^1(\pi_1(X))$ in terms of hyperbolic orbifold fibrations, and
feeding this into the main theorem shows that
$\Trop_{\k}(\VV^1(X,L_\sigma))$ is confined to the images of those
fibrations, for every $\sigma$
(Corollary~\ref{cor:bns-kahler-intro}).  Pushing the same argument
through Fox calculus on orbifold fundamental groups yields a sharp
dichotomy: the first twisted Alexander polynomial of a K\"{a}hler group
is always $0$ or $1$ (Theorem~\ref{thm:intro-kahler}).  This is a
genuinely new obstruction to K\"{a}hlerianity, not a repackaging of the
untwisted one---Example~\ref{ex:KT-knot product with WRRAG} exhibits
groups with $\Delta(G)=1$, so that the classical dichotomy of
\cite[Thm.~4.3(3)]{DPS-imrn} says nothing, but with
$\Delta^{\sigma}(G)\ne 1$ for a finite-image $\sigma$ over
$\overline{\F}_{13}$, so that Theorem~\ref{thm:intro-kahler} rules them
out.

\subsection{What is new in the proof}
\label{intro:new}
The bound of Theorem~\ref{thm:bns-twisted-trop-intro} is proved by
relating the twisted chain complex of $X$ to the Novikov--Sikorav
completion $\widehat{\Z G}_{-\chi}$.  The reason the untwisted argument
of \cite{PS-plms,Su-mathann} does not simply carry over to higher rank
is that an element of this completion is an \emph{infinite}\/ sum, and
there is no reason for such a sum to act on a rank-$r$ module: the
individual terms need not tend to zero.  Lemma~\ref{lem:nov-mod}
isolates the two conditions that force convergence---a uniform lower
bound
\begin{equation*}
\upsilon([\sigma(g)]_{ij}) \ge d \qquad \text{for all } g\in G,\
1\le i,j\le r,
\end{equation*}
on the valuations of the matrix entries of $\sigma$, together with
$\upsilon(\det(\sigma(g)))=0$ for all $g\in G$---and shows that under
them the completed tensor product $\hat{\K}^r$ is a left
$\widehat{\Z G}_{-\chi}$-module via $\sigma\otimes \rho$, with
$\chi=r(\upsilon\circ \rho)$.  This \emph{valuation bound}~\eqref{eq:val-bound}
is what lets higher-rank local systems enter the Novikov--Sikorav
framework at all, and it is the new input of the paper.  Granted it, the
K\"{u}nneth spectral sequence argument of \cite[Thm.~10.1]{PS-plms} and
\cite[Thm.~6.4]{Su-mathann} can be run.
Lemma~\ref{lem:novikov-conditions} supplies checkable conditions under
which the hypothesis holds, including the cases when $\sigma$ has finite
image or is defined over~$\Z$; these are exactly the cases used in the
applications.

\subsection{BNSR invariants and homology jump loci}
\label{intro:BNSR}
A landmark group-theoretic invariant was introduced in 1987 by Bieri, Neumann,
and Strebel \cite{BNS}---now known as the BNS invariant. This invariant generalizes
an earlier notion developed by Bieri and Strebel \cite{BS80,BS81} for metabelian groups.
Subsequently, Bieri and Renz \cite{BR}  extended the invariant to higher degrees, while
Farber, Geoghegan, and Sch\"utz \cite{FGS} adapted it from groups to topological spaces.
These extensions are collectively referred to as the Bieri--Neumann--Strebel--Renz (BNSR)
invariants, also known as $\Sigma$-invariants, which capture geometric finiteness properties
of groups and spaces.

Computing the BNSR invariants is notoriously challenging. Even in degree one, explicit descriptions
exist only for limited families of groups, including metabelian groups \cite{BS80,BS81,BG},
one‑relator groups \cite{Br}, right‑angled Artin groups \cite{MMV}, K\"ahler groups \cite{De10},
and pure braid groups \cite{KMM}. To address this difficulty, Papadima--Suciu
\cite{PS-plms} and Suciu \cite{Su-mathann} initiated a program to construct
upper bounds for the BNSR invariants via homology jump loci, which are more
computationally accessible.

The homology jump loci of a CW complex are defined using homology with coefficients in
rank‑one local systems. These loci naturally extend the notion of singular homology by
capturing how the homology varies as the local system ranges over a character variety.
Moreover, they are intimately related to Alexander‑type invariants, most notably the
Alexander polynomial from knot theory.

Papadima and Suciu \cite{PS-plms} first obtained an upper bound for the
BNSR invariants by means of the exponential tangent cones of the homology jump loci.
Subsequently, Suciu \cite{Su-mathann} refined this bound using the
tropical variety associated to the homology jump loci with complex coefficients. A crucial
advantage of the tropical approach is that the tropical variety detects positive‑dimensional
translated components of these jump loci that do not pass through the origin, whereas the
exponential tangent cone only captures the analytic germ at the origin of the jump loci.
More recently, Liu and Liu \cite{LL} further improved the bound by
employing the tropical variety associated to the homology jump loci with coefficients in
arbitrary fields as well as with integer coefficients. For metabelian groups, this bound
at degree one coincides with the bound obtained by Bieri, Groves, and Strebel \cite{BS80,BS81,BG}
and is therefore sharp. In general, however, these bounds are not sharp.

In knot theory, when the Alexander polynomial fails to resolve a given problem, a common
strategy is to consider non‑abelian invariants obtained by twisting with a linear representation
of the fundamental group---these are known as twisted Alexander polynomials. Inspired by
this idea, we introduce a twisted version of the homology jump loci, which generalizes the
twisted Alexander polynomials of knots. Using these twisted jump loci, we obtain sharper
tropical upper bounds for the BNSR invariants. Conversely, for spaces or groups whose
BNSR invariants are already understood---such as K\"ahler groups---this relationship
allows us to derive non‑trivial constraints on their twisted jump loci and twisted Alexander
polynomials.

\subsection{The main theorem}
\label{main:intro}
Let $X$ be a connected finite CW-complex with $\pi_1(X)=G$.
In this paper, \textit{we always assume that $\dim H^1(X; \R)>0$}.
Let $\rS(G)$ denote the unit sphere in $\Hom(G,\R) \cong H^1(X;\R)$,
and for a field $\k$ with a possibly trivial valuation $\upsilon$,
let $\Trop_{\k,\upsilon}(\VV^{\le q}(X,L_\sigma)) \subset H^1(X;\R)$
denote the tropicalization of the $q$-th twisted homology jump loci
associated to a representation $\sigma\colon G \to \GL_r(\k) $
(see Section~\ref{sect:twist-cjl} for precise definitions).
We work with both the homological invariants $\Sigma^q(X;\Z)$
of \cite{FGS} and the BNS set $\Sigma^1(G)$; recall
from \cite[\S 1.3]{BR} that these are related by
$\Sigma^1(G)=-\Sigma^1(G;\Z)$, as reviewed in
Section~\ref{sec:bnsr}. The sign is immaterial whenever the
invariant is antipodally symmetric, as it is for $3$-manifold
groups, but not in general; compare Example~\ref{ex:one-rel}.

The central result of the paper is the following tropical bound,
which strengthens the untwisted bounds of \cite{Su-mathann,LL} by
allowing the local system to range over representations of any rank.

\begin{theoremalph}[Theorem~\ref{thm:bns-twisted-trop}]
\label{thm:bns-twisted-trop-intro}
Let $X$ be a connected finite CW-complex with $\pi_1(X)=G$.
Fix a field $\k$ with a rational valuation
$\upsilon$ (i.e., $\upsilon(\k^{\times})\subseteq \Q$), possibly trivial.
Let $\sigma\colon G\to \GL_r(\k)$ be a representation whose matrix
entries have valuations uniformly bounded below and whose determinant
satisfies $\upsilon(\det(\sigma(g)))=0$ for all $g\in G$;
let $L_\sigma$ denote the corresponding local system on $X$.
Then
\[
\Sigma^q(X;\Z) \subseteq
\rS\big(\!\Trop_{\k,\upsilon}(\VV^{\le q}(X,L_{\sigma}))\big)^{\compl}.
\]
\end{theoremalph}

For $\sigma=\id$ and $r=1$ this is the bound of
\cite{Su-mathann,LL}.  The gain from higher rank is already visible for
$2$-generator, $1$-relator groups: Example~\ref{ex:one-rel} exhibits a
group for which the untwisted bound is nearly vacuous, excluding two
points of the circle $\rS(G)$, while the bound of
Theorem~\ref{thm:bns-twisted-trop-intro} applied to a rank-$2$
representation with image $S_3$ is an equality.

\begin{corollary}[Corollary~\ref{cor:twist-bound}]
\label{cor:twist-bound-intro}
Let $X$ be a connected finite CW-complex with $\pi_1(X)=G$.
For any representation $\sigma\colon G\to \GL_r(\Z)$,
\[
\Sigma^q(X;\Z) \subseteq
\rS\bigl(\!\Trop_\Z\bigl(\cJ_{\le q}(X,L_{\sigma})\bigr)\bigr)^{\compl}.
\]
\end{corollary}

Here $\cJ_{\leq q}(X,L_\sigma)$ denotes the twisted homology jump ideal over
$\Z$ (see Section~\ref{sect:twist-cjl}). Integrality of the matrix entries
of $\sigma$ automatically satisfies the valuation hypothesis of
Theorem~\ref{thm:bns-twisted-trop-intro} for all fields
simultaneously, so the $\Z$-tropical variety encodes all
field-valued tropicalizations at once.

\subsection{K\"ahler groups}
\label{apps:intro}

A finitely presented group $G$ is called a \emph{K\"ahler group}
if $G\cong \pi_1(X)$ for some compact K\"ahler manifold $X$.
For K\"ahler groups, the BNS invariant $\Sigma^1(G)$ is completely
described by Delzant's theorem \cite{De10} in terms of hyperbolic
orbifold fibrations of $X$. Combining Theorem~\ref{thm:bns-twisted-trop-intro}
with Delzant's theorem yields the following result, proved in Section~\ref{sect:Kahler}.

\begin{corollary}[Corollary~\ref{cor:bns-kahler}]
\label{cor:bns-kahler-intro}
Let $X$ be a compact K\"ahler manifold with $\pi_1(X)=G$.
For any algebraically closed field $\k$ with trivial valuation
and any representation $\sigma\colon G\to \GL_r(\k)$,
\[
\Sigma^1(G;\Z) \subseteq
\rS \Bigl(\Trop_\k\bigl(\VV^1(X,L_\sigma)\bigr)\Bigr)^{\!\compl},
\]
or equivalently, $\Trop_\k(\VV^1(X,L_\sigma))$ is contained in the
union of the images $\im(f^{*}\colon H^1(C_g;\R)\to H^1(X;\R))$ over
all hyperbolic orbifold fibrations $f\colon X\to C_g$.
\end{corollary}

This strengthens results of Suciu \cite{Su-mathann} and
Liu--Liu \cite{LL}, which treated the case $\sigma=\id$.
A further application of Theorem~\ref{thm:bns-twisted-trop-intro}
via Fox calculus on orbifold fundamental groups
gives the following new obstruction to K\"ahlerianity.

\begin{theoremalph}[Theorem~\ref{thm:alexpoly-kahler}]
\label{thm:intro-kahler}
Let $X$ be a compact K\"ahler manifold with $\pi_1(X)=G$.
Then for any representation $\sigma\colon G\to \GL_r(\k)$
over any field $\k$, the first twisted Alexander polynomial
$\Delta^\sigma(X)$ is either $0$ or $1$.
\end{theoremalph}

This strengthens \cite[Thm.~4.3(3)]{DPS-imrn}, which
established the same dichotomy for the \emph{untwisted}
Alexander polynomial $\Delta(G)$ (i.e., $\sigma = \id$, $r=1$).
The strengthening is not formal.  In
Example~\ref{ex:KT-knot product with WRRAG} we take $G$ to be the
product of the group of the Conway knot---whose Alexander polynomial is
$1$, but which by Friedl--Kim \cite{FK06} carries a representation into
$\GL_4(\overline{\F}_{13})$ factoring through $S_5$ with non-trivial
twisted Alexander polynomial---with a suitable weighted right-angled
Artin group.  For every such $G$ one has $\Delta(G)=1$, so
\cite[Thm.~4.3(3)]{DPS-imrn} gives no obstruction, whereas
$\Delta^{\sigma}(G)\ne 1$ for a finite-image $\sigma$, and
Theorem~\ref{thm:intro-kahler} shows that $G$ is not K\"ahler.  Thus
Theorem~\ref{thm:intro-kahler} detects non-K\"ahlerianity in cases
inaccessible to the untwisted result.

\subsection{$3$-manifold groups}
\label{apps:3mfd}

For a compact, connected, orientable $3$-manifold $M$ with
toroidal or empty boundary, the BNS invariant $\Sigma^1(\pi_1(M))$
is the projection of the open fibered faces of the Thurston norm
ball \cite[Thm.~E]{BNS}. Using a theorem of Friedl--Vidussi \cite{FV13},
which ensures that every non-fibered class is detected by the
twisted Alexander polynomial for some finite-image representation,
we show in Section~\ref{sect:3-manifold} that the twisted tropical
bound recovers $\Sigma^1$ exactly.

\begin{theoremalph}[Theorem~\ref{thm:sigma-3mfd}]
\label{thm:sigma-3mfd-intro}
Let $M$ be a compact, connected, orientable $3$-manifold with
toroidal or empty boundary, and set $G=\pi_1(M)$. Then
\[
\Sigma^1(G) = \rS\Bigg(\overline{\bigcup_{\sigma}
\Trop_\C\big(\VV^1(M,L_{\sigma}\otimes \C)\big)}\Bigg)^{\compl},
\]
where the union runs over all representations
$\sigma\colon G\to \GL_r(\Z)$ with finite image.
\end{theoremalph}

Theorem~\ref{thm:sigma-3mfd-intro} says that the inclusion of
Theorem~\ref{thm:bns-twisted-trop-intro} is an equality for
$3$-manifold groups, and so, in view of \cite[Thm.~E]{BNS}, that the
twisted tropical varieties of $M$ determine the fibered faces of the
Thurston norm ball.  The sharpness fails for the untwisted bound, since
untwisted Alexander polynomials do not detect fiberedness, whereas by
\cite{FV13} the twisted ones do.  In this sense the $3$-manifold case
shows that Theorem~\ref{thm:bns-twisted-trop-intro} is the right bound,
and not merely a better one.

\subsection{Further directions}
\label{intro:further}

The $3$-manifold result naturally raises the question of which
other spaces admit a sharp twisted tropical description of
their BNS invariants.

\begin{question}
\label{que:sigma}
For which spaces $X$ and degrees $q\geq 1$ do we have
\begin{equation}
\label{eq:sigma-bigcup}
\Sigma^q(X;\Z) = \rS\bigg(\overline{\bigcup_{\sigma}
\Trop_\Z\big(\cJ^{\le q}(X, L_{\sigma})\big)}\bigg)^{\compl},
\end{equation}
where $\sigma$ runs over all representations
$\sigma\colon G \to \GL_r(\Z)$?
In particular, does \eqref{eq:sigma-bigcup} hold when $q=1$?
\end{question}

The inclusion $\supseteq$ holds in general by
Corollary~\ref{cor:twist-bound-intro}, and
Theorem~\ref{thm:sigma-3mfd-intro} gives an affirmative
answer for $3$-manifold groups with $q=1$.  For right-angled
Artin groups and nilmanifolds, equality already holds with the
trivial representation alone \cite{PS-plms}, so no twisting is
needed in those cases.  Whether infinite-image integral
representations can yield sharper bounds than finite-image ones
is itself open; see Remark~\ref{rem:integral-vs-finite}.

\subsection{Organization}
\label{intro:organization}
This paper is organized as follows. In Section~\ref{sect:twist-cjl},
we introduce the twisted homology jump loci and twisted Alexander
polynomials and establish their basic properties.
In Section~\ref{sect:tropical}, we recall the necessary background
on tropical geometry.
After a brief review of the BNSR invariants in Section~\ref{sec:bnsr},
we prove Theorem~\ref{thm:bns-twisted-trop-intro}
in Section~\ref{sect:twist-bound}.
Sections~\ref{sect:Kahler} and~\ref{sect:3-manifold} are devoted
to the two applications described above:
Corollary~\ref{cor:bns-kahler-intro} and Theorem~\ref{thm:intro-kahler}
are both proved in Section~\ref{sect:Kahler}, and
Theorem~\ref{thm:sigma-3mfd-intro} is proved in
Section~\ref{sect:3-manifold}.

\section{Twisted homology jump loci}
\label{sect:twist-cjl}

This section introduces the main invariants studied in the paper.
We define the twisted homology jump loci in \S\ref{subsec:twist-cjl}, 
then develop the associated twisted Alexander polynomials in 
\S\ref{subsec:twist-alex}, including the Crowell exact sequence 
relating two natural module presentations via Fox calculus.
The section closes in \S\ref{subsec:cjl-covers} with a treatment 
of finite covers, which are a primary source of finite-rank 
local systems in our applications.

\subsection{Twisted homology jump loci}
\label{subsec:twist-cjl}
Let $X$ be a connected, finite-type CW-complex, with fundamental 
group $\pi_1(X)=G$. Fix a field $\k$, and let $\Hom(G,\GL_r(\k))$ be 
the $\k$-algebraic variety of rank~$r$ representations of $G$ over $\k$.  
Each representation $\sigma\colon G\to \GL_r(\k)$ corresponds to a rank $r$ 
local system $L_{\sigma}$ on $X$. The degree~$i$, rank~$r$ homology jump 
loci of $X$ (over $\k$) are the subvarieties 
\begin{equation}
\label{eq:cv-sigma}
\VV^i(X,\k)[r]=\big\{\sigma \in \Hom(G,\GL_r(\k)) \mid  
H_i(X; L_{\sigma}) \ne 0\big\}\, .
\end{equation}
In particular, the rank-one jump loci $\VV^i(X,\k)\coloneqq\VV^i(X,\k)[1]$ 
are subvarieties of $\Hom(G,\k^{\times})$, the commutative, 
$\k$-algebraic group of $\k$-valued, multiplicative characters of $G$.
If $\k\subset \K$ is a field extension, then $\VV^i(X,\k)[r]$ are the 
$\k$-points on $\VV^i(X,\K)[r]$, while $\VV^i(X,\K)[r]$ is obtained from 
$\VV^i(X,\k)[r]$ by extension of the base field. 

Now fix a local system $L_\sigma$ on $X$ over $\k$ of rank $r\geq 1$. 
Observe that if $\rho\colon G\to \k^{\times}$ is a character, then the 
local system $L_{\sigma}\otimes_{\k} L_{\rho}$ corresponds to 
the representation $\sigma\otimes \rho\colon G\to \GL_r(\k)$ that scales each 
matrix $\sigma(g)$ by the scalar $\rho(g)$. This defines an action of the 
character group $\Hom(G,\k^{\times})$ on the representation variety 
$\Hom(G,\GL_r(\k))$. We are interested in the orbit of $L_\sigma$ 
for this action, in particular, its intersection with the subvariety 
$\VV^i(X,\k)[r]$ defined in \eqref{eq:cv-sigma}. The {\em twisted 
homology jump loci}\/ of $X$ with coefficients in 
a local system $L_\sigma$ (in degree $i$) are the jump loci for 
homology with coefficients in rank-one local systems, twisted by $L_\sigma$, 
\begin{equation}
\label{eq:cvx}
\VV^i(X,L_\sigma)=\big\{\rho \in \Hom(G,\k^{\times}) \mid  
H_i(X; L_\sigma\otimes_{\k} L_{\rho}) \ne 0\big\}\, .
\end{equation}
We will denote by $\VV^{\le q}(X,L_\sigma)=\bigcup_{i\le q} \VV^i(X,L_\sigma)$ 
the union of the jump loci in degrees up to $q$. 

Note that when $L_\sigma$ is the constant $\k$-sheaf, 
$\VV^i(X,L_\sigma)=\VV^i(X,\k)$. 
Moreover, if $L_\sigma$ is a rank-one local system, then  
$\sigma$ is a character in $\Hom(G,\k^{\times})$. Hence 
$L_\sigma\otimes_{\k} L_{\rho}=L_{\sigma\cdot \rho}$, and so   
\begin{equation}
\label{eq:rho-vv}
\VV^i(X,L_{\sigma\cdot \rho})=\rho^{-1} \cdot \VV^i(X,L_\sigma) .
\end{equation}

\begin{proposition}
\label{prop:product-formula}
 Let $X_1$ and $X_2$ be two connected, finite-type CW-complexes. 
 Consider two local systems $L_{\sigma_1}$ and $L_{\sigma_2}$ on $X_1$ 
 and $X_2$ both over a field $\k$, respectively.  Then we have
\[
\VV^i(X_1 \times X_2, L_{\sigma_1} \boxtimes L_{\sigma_2})= 
\bigcup_{s+t=i} \VV^s(X_1, L_{\sigma_1}) \times \VV^t(X_2, L_{\sigma_2}) .
\]
Here $L_{\sigma_1} \boxtimes L_{\sigma_2} = p_1^{-1}L_{\sigma_1} 
\otimes p_2^{-1}L_{\sigma_2}$ with $p_1$ and $p_2$ being the projection 
maps from $X_1\times X_2$ to $X_1$ and $X_2$, respectively.  
\end{proposition}

\begin{proof}
Given a character $\rho = (\rho_1,\rho_2) \in \Hom(\pi_1(X_1\times X_2),\k^\times)=
\Hom(\pi_1(X_1),\k^\times)\times \Hom(\pi_1( X_2),\k^\times) $, by the K\"unneth formula 
(the $\Tor$ term vanishes since $\k$ is a field), we have 
\[
H_i(X_1\times X_2; (L_{\sigma_1} \boxtimes L_{\sigma_2})\otimes_{\k} 
L_\rho) \cong \bigoplus_{s+t=i} H_s(X_1; L_{\sigma_1} \otimes L_{\rho_1})\otimes_\k  
H_t(X_2; L_{\sigma_2} \otimes L_{\rho_2}).
\]
The claim follows.
\end{proof}

\subsection{Equations for twisted homology jump loci} 
\label{subsec:twist-det}
Here is a more concrete description of the (twisted) homology jump loci. 
First let $\widetilde{X}\to X$ be the universal cover of $X$. Upon lifting 
the cell structure of $X$ to this cover, we obtain a chain complex of free 
$\Z[G]$-modules, $(C_{*}(\widetilde{X};\Z), \partial)$.

Let $R$ be $\Z$ or a field $\k$. Fix a representation $\sigma\colon G \to \GL_r(R)$. 
Consider the chain complex 
\begin{equation}
\label{eq:twist-cc}
C_*^\sigma(X ; R) \coloneqq (C_{*}(\widetilde{X};\Z) \otimes_{\Z[G]} 
(R[G_{\ab}] \otimes R^r), \partial^{\sigma})
\end{equation}
where $\ab\colon G\surj G_{\ab}$ is the abelianization of $G$ and 
$G$ acts on $R[G_{\ab}] \otimes R^r$ via 
\[
g \cdot (u\otimes v) = (\ab(g)\cdot u)\otimes (\sigma(g) \cdot v)
\]
and $\partial^{\sigma} =\partial \otimes \id$. 
The chain complex \eqref{eq:twist-cc} is a complex of finitely 
generated free $R[G_{\ab}]$-modules. 
The {\em $i$-th twisted homology jump ideal}\/ $\cJ_i(X,L_\sigma)$ is the ideal
generated by minors of size $r\cdot c_i$ of the matrices 
$\partial^\sigma_{i+1} \oplus \partial^\sigma_{i}$, where $c_i$ 
is the number of $i$-cells of $X$. When $R=\k$ is an algebraically 
closed field, the varieties $\VV^i(X,L_\sigma)$, then, are the 
zero-sets of the $i$-th twisted homology jump ideal $\cJ_i(X,L_\sigma)$.
When $i=1$, the set $\VV^1(X,L_{\sigma})$ only depends on 
$G=\pi_1(X)$ and $\sigma$, so it can also be denoted by $\VV^1(G,L_{\sigma})$.
Moreover, $\VV^1(G,L_{\sigma})$ can be computed by means of the Fox calculus.

\subsection{Twisted Alexander polynomials}
\label{subsec:twist-alex}

Let $\sigma\colon \pi_1(X)\to \GL_r(R)$ be a representation, where $R$ is 
either $\Z$ or a field $\k$. Fix a surjective homomorphism 
$\phi\colon \pi_1(X) \twoheadrightarrow A$, where $A$ is a 
finitely generated free abelian group. 

Consider the chain complex
\begin{equation}
\label{eq:twist-cc2}
C_*^{\phi,\sigma}(X ;R) \coloneqq 
C_{*}(\widetilde{X};\Z) \otimes_{\Z[G]} (R[A] \otimes_R R^r),
\end{equation}
where $G=\pi_1(X)$ acts on $R[A] \otimes R^r$ via 
\[
g \cdot (u\otimes v) = (\phi(g)\cdot u)\otimes (\sigma(g) \cdot v).
\]
This is a chain complex of finitely generated free $R[A]$-modules. 
For $i\geq 0$, the \emph{$i$-th twisted Alexander invariant} of 
$(X,\phi,\sigma)$ is
\begin{equation}
\label{eq:twist-alex-mod}
H_i^{\phi,\sigma}(X;R[A])\coloneqq H_i\bigl(C_*^{\phi,\sigma}(X ;R)\bigr),
\end{equation}
a finitely generated $R[A]$-module. Since $R[A]$ is a Noetherian UFD, 
each such module admits a finite presentation
\[
\begin{tikzcd}[column sep=18pt]
R[A]^m \ar[r, "M"] &[4pt] R[A]^n \ar[r] &H_i^{\phi,\sigma}(X;R[A]) \ar[r] & 0.
\end{tikzcd}
\]
The \emph{$k$-th Fitting ideal} of a finitely presented $R[A]$-module 
with presentation matrix $M\colon R[A]^m\to R[A]^n$ is the ideal
$\Fitt_k(M)$ generated by the $(n-k)\times(n-k)$ minors of $M$,
with the convention $\Fitt_k(M)=R[A]$ if $k\geq n$ and 
$\Fitt_k(M)=0$ if $n-k>m$.

\begin{definition}
\label{def:twist-alex-poly}
The \emph{$i$-th (twisted) Alexander polynomial} $\Delta_i^{\phi,\sigma}(X)\in R[A]$ 
is the gcd of the elements of $\Fitt_0\!\bigl(H_i^{\phi,\sigma}(X;R[A])\bigr)$;
it is well-defined up to units in $R[A]$.
For $i=1$, we write $\Delta^{\phi,\sigma}(X)$, and if $\phi$ is the projection 
onto the maximal free abelian quotient of $\pi_1(X)$, we denote it by 
$\Delta^\sigma(X)$.
\end{definition}

Now assume $R=\k$ is an algebraically closed field. The epimorphism $\phi$ 
induces an embedding
\[
\phi^* \colon \Hom(A,\k^{\times}) \longinj \Hom(G,\k^{\times}).
\]
If $\rank A>1$, the irreducible factors of $\Delta^{\phi,\sigma}(X)$ correspond 
bijectively to the codimension-one irreducible components of
\[
\VV^1(X,L_\sigma) \cap \im \phi^*
\subseteq \im \phi^* \cong (\k^\times)^{\rank A}.
\]
In the untwisted case, this correspondence was established in \cite[Cor.~3.2]{DPS-imrn}. 
The following proposition extends this relationship to the twisted setting.

\begin{proposition} 
\label{prop:jump-loci-and-Alexander}
Let $\sigma\colon \pi_1(X)\to \GL_r(\k)$ be a representation over an 
algebraically closed field $\k$.
\begin{enumerate}
\item \label{ja1}
If $\rank A>1$, the zero locus of $\Delta^{\phi,\sigma}(X)$ coincides with the 
union of all codimension-one irreducible components of 
$\VV^1(X,L_\sigma) \cap \im \phi^*$. Moreover,
\[
\Delta^{\phi,\sigma}(X)=0 \;\Longleftrightarrow\; 
\im \phi^* \subseteq \VV^1(X,L_\sigma).
\]

\item \label{ja2}
If $\rank A=1$, the module $H_1^{\phi,\sigma}(X;\k[A])$ is not torsion 
(equivalently, $\Delta^{\phi,\sigma}(X)=0$) if and only if 
\[
\im \phi^* \subseteq \VV^1(X,L_\sigma).
\]
Moreover, writing $\k[A]\cong \k[t^{\pm1}]$, the free rank of 
$H_1^{\phi,\sigma}(X;\k[A])$ equals $m$ if and only if 
\[
\dim_\k H_1\bigl(X, L_\sigma \otimes_{\k} \phi^* L_\lambda\bigr)=m
\]
for generic $\lambda\in \k^\times$.
\end{enumerate}
\end{proposition}

\begin{proof}
The arguments from \cite[Cor.~3.2]{DPS-imrn} carry over to the twisted 
setting, yielding all claims except the last assertion in \eqref{ja2}, namely, 
the formula relating free rank to generic fiber dimension.
(Although \cite{DPS-imrn} uses the Fox-calculus formula for the Alexander 
polynomial, the proof there relies only on the identification of the 
characteristic variety with the support of $H_1^{\phi,\sigma}(X;\k[A])$ 
and the fact that the zero locus of $\Delta^{\phi,\sigma}(X)$ 
equals the union of codimension one components of the support of $H_1^{\phi,\sigma}(X;\k[A])$,
which holds equally for Definition~\ref{def:twist-alex-poly}.)

It remains to prove that the free rank of $H_1^{\phi,\sigma}(X;\k[A])$ 
equals $m$ if and only if $\dim_\k H_1(X; L_\sigma\otimes_\k\phi^*L_\lambda)=m$ 
for generic $\lambda\in\k^\times$.
When $\rank A=1$, the ring $\k[A]\cong \k[t^{\pm1}]$ is a PID, 
so the chain complex $C_*^{\phi,\sigma}(X;\k)$ decomposes up to 
chain homotopy as a direct sum of complexes of the form
\[
\begin{tikzcd}[column sep=18pt]
0 \ar[r]& \k[t^{\pm1}] \ar[r, "f"] &[2pt] \k[t^{\pm1}] \ar[r]& 0
\end{tikzcd}
\quad\text{or}\quad
\begin{tikzcd}[column sep=18pt]
0 \ar[r]& \k[t^{\pm1}] \ar[r]& 0
\end{tikzcd}
\]
with $f\in \k[t^{\pm1}]$.  The free rank of $H_1^{\phi,\sigma}$ equals
the number $m$ of summands of the second type.  Specializing at
$\lambda\in\k^\times$ (i.e., applying $-\otimes_{\k[t^{\pm1}],t\mapsto\lambda}\k$),
each second-type summand contributes $1$ to 
$\dim_\k H_1(X;L_\sigma\otimes_\k\phi^*L_\lambda)$, while each first-type 
summand contributes $1$ only if $f(\lambda)=0$, and $0$ otherwise.
Since there are finitely many first-type summands, $f(\lambda)\ne 0$ 
for all but finitely many $\lambda$, so for generic $\lambda$ the 
dimension equals $m$, as claimed.
\end{proof}

We now describe a twisted analogue of the classical Crowell exact 
sequence from Fox calculus.

\begin{lemma}
\label{lem:twisted-crowell}
For a finitely presented group $G=\langle x_1,\dots,x_n\mid r_1,\dots,r_m\rangle$ 
and a representation $\sigma\colon G\to\GL_r(R)$, the boundary map 
$\partial_2^\sigma$ of the twisted chain complex fits into a short exact sequence
\begin{equation}
\label{eq:crowell}
\begin{tikzcd}[column sep=18pt]
0 \ar[r] & H_1^{\phi,\sigma}(X;R[A]) \ar[r] &A_\phi^\sigma 
\ar[r] & \im(\partial_1^\sigma) \ar[r] & 0.
\end{tikzcd}
\end{equation}
where $A_\phi^\sigma \coloneqq \coker(\partial_2^\sigma)$ is the 
\emph{twisted Alexander module}.
\end{lemma}

\begin{proof}
Both $H_1^{\phi,\sigma}(X;R[A])=\ker(\partial_1^\sigma)/\im(\partial_2^\sigma)$ 
and $A_\phi^\sigma=C_1^{\phi,\sigma}/\im(\partial_2^\sigma)$ are quotients 
of $C_1^{\phi,\sigma}$ by $\im(\partial_2^\sigma)$. The natural inclusion 
$H_1^{\phi,\sigma}\hookrightarrow A_\phi^\sigma$ is therefore well-defined, 
with cokernel $C_1^{\phi,\sigma}/\ker(\partial_1^\sigma)\cong\im(\partial_1^\sigma)$.
\end{proof}

\begin{remark}
\label{rem:fox-vs-invariant}
By the analogue of \cite[Prop.~2.2]{DPS-imrn} for 
Lemma~\ref{lem:twisted-crowell},
\[
V\!\left(\Fitt_0\bigl(H_1^{\phi,\sigma}(X;R[A])\bigr)\right) 
= V\bigl(\Fitt_r(A_\phi^\sigma)\bigr)
\]
outside finitely many characters in $\Hom(A,\k^\times)$, namely, 
the support of $\im(\partial_1^\sigma)$, which in the 
untwisted case reduces to $\{\bo\}$. 
The minors of size $(n-1)r$ in the presentation matrix of  $A_\phi^\sigma$   
generate $\Fitt_r(A_\phi^\sigma)$. 
Hence the polynomial $\Delta^\sigma(X)$ and the gcd of those minors 
define the same hypersurface outside finitely many characters; 
this is why $\partial_2^\sigma$ can be used to compute 
$\VV^1(X,L_\sigma)$ via 
Proposition~\ref{prop:jump-loci-and-Alexander}.
\end{remark}

\subsection{Finite covers and twisted jump loci}
\label{subsec:cjl-covers}

Finite-index subgroups $N\leq G$ yield finite-rank local systems
defined over $\Z$: the action of $G$ on the left cosets of $N$
gives a representation $G\to S_r$, where $r=[G\colon N]$.

The most geometric instance is a connected finite cover
$\pi\colon Y\to X$, with $G=\pi_1(X)$, $N=\pi_1(Y)$, and $r=[G\colon N]$.
The pushforward $L_\pi\coloneqq\pi_* R_Y$ is a rank-$r$ $R$-local system
on $X$ induced by this coset action. When $N\trianglelefteq G$---equivalently, 
when $Y$ is a regular cover---$L_\pi$ is associated to the regular 
representation of $G/N$.

The next proposition records a formula for computing the homology 
groups $H_i(Y;\k)$ from the twisted homology jump loci.

\begin{proposition} 
\label{prop:finite-cover}
Let $\pi\colon Y\to X$ be a connected finite regular cover of a connected 
finite CW-complex $X$. Set $\pi_1(Y)=N$ and $\pi_1(X)=G$.
Assume that $\ch(\k) \nmid r=\vert G/N \vert$ and $\k$ is algebraically closed. 
Let $L_1,\dots, L_s$ be the irreducible $\k$-representations of the 
finite group $G/N$, with $L_1$ the trivial representation and $s$ the 
number of conjugacy classes of $G/N$.
\begin{enumerate}[itemsep=3pt]
\item \label{fc-homology}
For all $i\ge 0$, there is a decomposition
\begin{equation}
\label{eq:decomp}
\begin{aligned}
H_i(Y;\k) & \cong \bigoplus_{1\leq j \leq s}  
H_i(X;L_j)^{\oplus \rank L_j} \\
& \cong H_i(X;\k) \oplus \bigoplus_{2\leq j \leq s}
H_i(X;L_j)^{\oplus \rank L_j}.
\end{aligned}
\end{equation}

\item \label{fc-varieties}
The homomorphism $\pi^* \colon \Hom(\pi_1(X),\k^{\times})\to 
\Hom(\pi_1(Y),\k^{\times})$ satisfies
\begin{equation}
\label{eq:pullback-cv}
(\pi^*)^{-1}\big(\VV^i(Y,\k)\big) 
= \bigcup_{j=1}^{s} \VV^i(X,L_j)
\end{equation}
for all $i\ge 0$. That is, a character 
$\rho\in \Hom(\pi_1(X),\k^{\times})$ has the property that 
$\pi^*\rho \in \VV^i(Y,\k)$ if and only if 
$\rho \in \VV^i(X,L_j)$ for some $1\le j\le s$.
\end{enumerate}
\end{proposition}

\begin{proof}
Since $\ch(\k) \nmid r=\vert G/N \vert$, by Maschke's theorem, 
$L_{\pi}$ is a semi-simple local system. 
Moreover, since $L_\pi$ corresponds to the regular representation 
of the finite group $G/N$, we have $L_\pi \cong 
\bigoplus_{1\leq j \leq s} L_j^{\oplus \rank L_j}$ by the 
decomposition theorem for regular representations of 
finite groups. Part~\eqref{fc-homology} follows.

For part~\eqref{fc-varieties}, let $L_\rho$ be any rank-one 
$\k$-local system on $X$. By the projection formula, 
\begin{equation}
\label{eq:proj-formula}
H_i(Y; \pi^*L_\rho) \cong H_i(X; L_\pi \otimes_\k L_\rho)
\cong \bigoplus_{1\leq j \leq s} 
H_i(X;L_j\otimes_\k L_\rho)^{\oplus \rank L_j}.
\end{equation}
Hence $\pi^*\rho \in \VV^i(Y,\k)$, i.e., 
$H_i(Y;\pi^*L_\rho)\ne 0$, if and only if 
$H_i(X;L_j\otimes_\k L_\rho)\ne 0$ for some $j$, 
which is equivalent to $\rho \in \VV^i(X,L_j)$ 
for some $1\le j\le s$.
\end{proof}

\begin{remark}
\label{rem:equivalence}
Call two irreducible representations $L$ and $L'$ of $G/N$ 
\emph{equivalent} if $L \cong L' \otimes_\k L_\rho$ for some 
rank-one representation $L_\rho$ of $G/N$; in which case
$\VV^i(X,L) = \rho^{-1}\cdot\VV^i(X,L')$.
Since $\k^\times$ is abelian, every rank-one representation of 
$G/N$ factors through $\ab(G/N)$, so the equivalence classes 
are orbits of $\Hom(\ab(G/N),\k^\times)$ acting on the 
irreducible representations of $G/N$.
\end{remark}

\section{Tropical varieties}
\label{sect:tropical}
In this section, we focus on tropical geometry, where the base ring $R$ 
is a field $\k$ or $\Z$. The reader may refer to the monograph \cite{MacStu} 
for the required background on tropical geometry.

\subsection{Tropical variety over a valued field}
\label{subsec:trop-valued-field}
Let $\k$ be a fixed field endowed with a possibly trivial valuation 
$\upsilon\colon\k^\times\to \R$ and $\upsilon(0)=\infty$. 
Consider the group $H=\Z^n$. Fixing a basis, we may identify 
the group ring $\k H$ with the Laurent polynomial ring 
$\k[x_1^{{\pm 1} },\ldots,x_n^{{\pm 1} }]$.

\begin{definition}
\label{def:tropical-variety}
For an ideal $I\subseteq \k[x_1^{{\pm 1} },\ldots,x_n^{{\pm 1} }]$, 
we recall the following two  ways to define the \textit{tropical variety of $I$}. 

\begin{enumerate} 
\item \label{trop1} Fix $\mathbf{w}\in \R^n$. For any nonzero polynomial 
$f=\sum\limits_{\textbf{u}\in\Z^n}a_{\textbf{u}} 
x^{\textbf{u}}\in \k[x_1^{{\pm 1} },\ldots,x_n^{{\pm 1} }]$,  the initial form 
$\init_{\textbf{w},\upsilon}(f)$ is the sum of all terms in $f$ 
where
\[
\min\limits_{{\textbf{u}}\in \Z^n} 
\left\{ \upsilon(a_{{\bf{u}}})+{\bf{u}}\cdot{\textbf{w}} \mid a_{\textbf{u}}\neq 0\right\}
\]
is achieved. The initial 
ideal $\init_{\textbf{w},\upsilon}(I)$ is the ideal generated by 
$\init_{\textbf{w},\upsilon}(f)$, where $f$ runs over $I$. Set 
\begin{equation*}
\label{def:trop-ideal}
\Trop_{\k, \upsilon}(I)\coloneqq \{\textbf{w}\in \R^n\mid 
\init_{\textbf{w},\upsilon}(I) \neq \k[x_1^{{\pm 1}},\ldots,x_n^{{\pm 1}}]\}.
\end{equation*}

\item \label{trop2} Let $\K$ be an algebraically closed field extending $\k$ 
such that the extension of $\upsilon$ to $\K$ is nontrivial, still denoted 
$\upsilon$. Such a field always exists and the specific choice of $\K$ is 
not important, as long as it is algebraically closed with a nontrivial 
valuation (see \cite[Thm.~3.2.4 and Rem.~3.2.5]{MacStu}). 
Denote by $V(I)\subset (\K^{\times})^n$ the zero locus of $I\otimes_{\k}\K$.
The tropical variety of $I$ is then defined as
\[
\Trop_{\k,\upsilon}(I) \coloneqq 
\overline{\{(\upsilon(x_1),\ldots,\upsilon(x_n))\mid 
(x_1,\ldots,x_n)\in V(I)\}},
\]
the closure in the Euclidean topology of the image of $V(I)$ 
under the coordinatewise valuation map $(\K^{\times})^n \to \R^n$.
\end{enumerate}
\end{definition}

The  fundamental theorem of tropical algebraic geometry in 
\cite[Thm.~3.2.3]{MacStu} shows the equivalence of the above two definitions.
Moreover, $\Trop_{\k,\upsilon}(I)$ only depends on $\sqrt{I}$. 
If $\sqrt{I}=\bigcap\limits_{j=1}^q P_j$, where $\{P_j\}_{j=1}^q$ 
are all minimal prime ideals containing $I$, then
\[
\Trop_{\k,\upsilon}(I)=\Trop_{\k,\upsilon}\left(\!\sqrt{I}\,\right)=
\bigcup_{j=1}^q \Trop_{\k,\upsilon}(P_j).
\]

We recall the following structure theorem \cite[Thm.~3.3.8]{MacStu}.

\begin{theorem}
\label{thm:structure-trop}
Let $I$ be a prime ideal in $\k[x_1^{{\pm 1}},\ldots,x_n^{{\pm 1}}]$.
Then $\Trop_{\k,\upsilon}(I)\subset \R^n$ is a 
$\upsilon(\k^{\times})$-rational polyhedral complex, 
pure of dimension $\dim V(I)$.
\end{theorem}

When $H$ has nontrivial torsion, the character variety 
$\Hom(H,\K^\times)$ is a finite disjoint union $\coprod(\K^{\times})^n$,
whose components are translates of the identity component 
$(\K^{\times})^n$ by torsion characters; since torsion characters 
have valuation $0$, these translations are invisible to 
tropicalization. Writing $V_1,\ldots,V_q$ for the connected 
components of $V(I)$, each $V_j$ may be identified with a 
subset of $(\K^{\times})^n$ via such a translation, and
\[
\Trop_{\k,\upsilon}(I) \coloneqq \bigcup_{j=1}^{q} 
\overline{\{(\upsilon(x_1),\ldots,\upsilon(x_n))\mid 
(x_1,\ldots,x_n)\in V_j\}}.
\]

\subsection{Tropical varieties over rings} 
\label{subsec:trop-over-ring}
In this subsection, we follow \cite[Sec.~1.6]{MacStu} to define 
the tropical variety over a commutative Noetherian ring $R$. 
We refer to \cite[Sec.~3]{LL} and \cite{BG} for more details.

Let $H$ be a finitely generated abelian group, not necessarily free. 
Fix an additive character $\chi \in \Hom(H ;\R)$. 
For any nonzero $f=\sum_{h\in H} a_h h \in R H$, we denote by $\deg_\chi(f)$ 
the minimal value of $\chi(h)$ with $a_h\neq 0$, and call it the $\chi$-degree of $f$. 
The initial form $\init_\chi(f)$ is the sum of all nonzero terms $a_h h$ in $f$ 
such that $\chi(h)=\deg_\chi(f)$. For an ideal $I \subseteq RH$, its initial ideal with 
respect to $\chi$ is defined as
\[
\init_\chi (I)\coloneqq \langle\init_\chi(f)\mid f\in I\rangle.
\]
Following \cite[Sec.~1.6]{MacStu} and Definition \ref{def:tropical-variety}\eqref{trop1}, 
one can define the tropical variety over $R$ as follows. 

\begin{definition}
\label{def:tropical-ring}
The {\it tropical variety over $R$} of an ideal $I\subseteq RH $ 
is the following set 
\[
\Trop_R(I)\coloneqq\{\chi \in \Hom(H ;\R) \mid 
\init_\chi (I)\neq R H \}.
\]
\end{definition}

Note that for the zero character $\init_{0} I=I$.
Hence $0\in \Trop_R(I)$ if and only if $I$ is a proper ideal in $RH$, 
which implies that $\Trop_R (I)=\emptyset$ if and only if $I=RH$. 
From now on we always assume that $I$ is a \textit{proper} ideal. 
Moreover, by definition if $\chi\in\Trop_R (I)$, then $r\cdot\chi\in\Trop_R (I)$ 
for any positive real number $r$. Therefore, $\rS(\Trop_R(I))$ 
carries the same information as $\Trop_R(I)$. Here $\rS(V)$ 
denotes the intersection of the unit sphere in $H^1(H;\R)$ 
with a subset $V \subset H^1(H;\R)$. 

When $R$ is a field $\k$, the tropical variety $\Trop_\k(I)$ defined in 
Definition \ref{def:tropical-ring} is indeed  $\Trop_{\k, \upsilon_0}(I)$ 
with the trivial valuation $\upsilon_0$ on $\k$ defined in 
Definition \ref{def:tropical-variety}. When $R=\Z$, the tropical 
variety $\Trop_\Z(I)$ can be understood using the 
tropical varieties over various fields, as follows. 

\begin{proposition}[{\cite[Prop.~3.9]{LL}}]
\label{prop:three-trop-union-Z}
Let $I$ be an ideal in $\Z H$. Then $\Trop_\Z(I)=\Trop_\Z(\sqrt{I})$. 
Moreover, 
\[
\rS\big(\Trop_\Z(I)\big) = 
\rS\Big(\Trop_{\Q,\upsilon_0}(I\otimes_\Z \Q)\cup\bigcup\limits_{p \in \PP}
\big(\Trop_{\Q,\upsilon_p}(I\otimes_\Z \Q)\cup \Trop_{\F_p,\hat{\upsilon}_p}(I\otimes_\Z \F_p)\big)
\Big),
\]
where $\PP$ is a finite set of primes (depending on $I$), and 
in the union we are taking the trivial valuation $\upsilon_0$ on $\Q$, 
the $p$-adic valuation $\upsilon_p$ on $\Q$ and the trivial valuation 
$\hat{\upsilon}_p$ on $\F_p$, respectively. 
\end{proposition} 

\begin{remark}
\label{rem:tropical over Z}
In general, we have
\[
\Trop_\Z(I) \neq \Trop_{\Q,\upsilon_0}(I\otimes_{\Z}\Q)\cup\bigcup\limits_{p \in \PP}
\big(\Trop_{\Q,\upsilon_p}(I\otimes_{\Z}\Q)\cup \Trop_{\F_p,\hat{\upsilon}_p}(I\otimes_{\Z}\F_p)\big),
\]
although Proposition \ref{prop:three-trop-union-Z} shows that 
the two sets coincide after projection onto the unit sphere.   
\end{remark}

\subsection{Tropical varieties of twisted jump loci}
\label{subsec:twist-jump-trop}
Let $X$ be a connected finite CW complex with $\pi_1(X)=G$ and $b_1(X)>0$.
Set $H=\ab(G)$. Let $\sigma\colon \pi_1(X)\to \GL_r(R)$ be a representation 
with $R$ being $\Z$ or a field $\k$.
Set $\cJ_{\leq q}(X,L_\sigma)\coloneqq \bigcap_{0\leq i \leq q} \cJ_{ i}(X,L_\sigma)$, 
which is an ideal in $R[H]$. If $R=\Z$, we get the 
tropical variety $\Trop_\Z( \cJ_{\leq q}(X,L_\sigma) )$.
If $R=\k$ is a field, then $\k$ is a field extension of $\Q$ or $\F_p$.
In the first case, by Ostrowski's theorem, every nontrivial valuation 
on $\Q$ is equivalent to the $p$-adic valuation $\upsilon_p$ for some 
prime $p$, see \cite[Thm.~3.1.4]{Go}; in addition, one may also take 
the trivial valuation $\upsilon_0$ on $\k$.
In the latter case ($\k$ a field extension of $\F_p$), 
the only valuation on $\F_p$ is the trivial valuation $\hat{\upsilon}_p$.
Moreover, if $R=\Z$, Proposition \ref{prop:three-trop-union-Z} shows 
that one can compute $\rS(\Trop_\Z( \cJ_{\leq q}(X,L_\sigma) ))$ by the 
projections of three types of tropical varieties over fields with the above 
mentioned valuations.

\begin{remark}
\label{rem:tropical-twisted-cvs}
With notation as in Proposition~\ref{prop:finite-cover} and 
Remark~\ref{rem:equivalence}: since $G/N$ is finite, every 
rank-one representation $L_\rho$ of $G/N$ has finite order, so 
$\upsilon(\rho)=0$ for any valuation $\upsilon$ on $\k$. 
Equivalent representations therefore have the same tropicalization, 
and
\[
\Trop_{\k,\upsilon}(\VV^i(X,L_\pi))=\bigcup_L \Trop_{\k,\upsilon}(\VV^i(X,L)),
\]
where $L$ runs over one representative from each equivalence class 
of irreducible $\k$-representations of $G/N$.
\end{remark}

\section{Bieri--Neumann--Strebel--Renz invariants}
\label{sec:bnsr}

In this section, we review the definition of the $\Sigma$-invariants 
of a group $G$ and, more generally, of a space $X$, following 
the approach from \cite{Bi07, FGS, PS-plms, Su-pisa, Su-mathann}.

\subsection{The \texorpdfstring{$\Sigma$}{Sigma}-invariants of a chain complex}
\label{subsec:bnsr cc}
Let $C=(C_i,\partial_i)_{i\ge 0}$ be a chain complex over a ring $R$, 
and let $q$ be a positive integer. We say $C$ is of {\em finite $q$-type}\/ 
if there is a chain complex $C'$ of finitely generated, projective 
(left) $R$-modules and a chain map $C'\to C$ inducing isomorphisms 
$H_i(C')\to H_i(C)$ for $i<q$ and an epimorphism 
$H_q(C')\to H_q(C)$. For a free chain complex $C$, this is equivalent 
to being chain-homotopy equivalent to a free chain complex $D$ 
for which $D_i$ is finitely generated for all $i\leq q$.

Now let $G$ be a finitely generated group.  The character sphere 
\begin{equation}
\label{eq:sg}
\rS(G)\coloneqq (\Hom(G,\R)\setminus\{0\})/\R^{+},
\end{equation}
is the set of nonzero homomorphisms $G\to \R$ modulo 
homothety. To simplify notation, we will usually denote both 
a nonzero homomorphism $\chi\colon G\to \R$ and its equivalence 
class, $[\chi] \in \rS(G)$, by the same symbol, $\chi$. The character 
sphere may be identified with the unit sphere $S^{n-1}$ in the 
real vector space $\Hom(G,\R)\cong \R^n$, where $n=b_1(G)$.  

Given a nonzero homomorphism $\chi\colon G\to \R$, the set  
$G_{\chi}\coloneqq\{ g \in G \mid \chi(g)\ge 0\}$ 
is a submonoid of $G$, which depends only on  
$[\chi]\in \rS(G)$. The monoid ring 
$\Z{G}_{\chi}$ is a subring of the group ring $\Z{G}$; 
thus, any $\Z{G}$-module naturally acquires the structure of a 
$\Z{G}_{\chi}$-module, by restriction of scalars.

\begin{definition}[\cite{FGS}]  
\label{def:sigma chain}
Let $C$ be a chain complex over $\Z{G}$.  For each 
integer $q\ge 0$, the {\em $q$-th 
Bieri--Neumann--Strebel--Renz invariant}\/ of $C$ 
is the set
\begin{equation}
\label{eq:sigmakc}
\Sigma^q(C)=\big\{\chi\in \rS(G) \mid 
\text{$C$ is of finite $q$-type over $\Z{G}_{\chi}$}\big\}\, .
\end{equation}
\end{definition}

Suppose now that $N\triangleleft G$ is a normal 
subgroup such that the quotient group $G/N$ is abelian.  Let  
$\rS(G,N)$ be the set of homomorphisms $\chi \in \rS(G)$ 
for which $N\le \ker (\chi)$. Then $\rS(G,N)$ 
is a great subsphere of $\rS(G)$, obtained by intersecting 
the unit sphere $\rS(G)\subset H^1(G;\R)$ 
with the image of the linear map 
$\kappa^*\colon H^1(G/N;\R) \inj H^1(G;\R)$, 
where  $\kappa\colon G\surj G/N$ is the canonical projection. 
Finally, let $C$ be a chain complex of free $\Z{G}$-modules, with 
$C_i$ finitely generated for $i\le q$, and let $N$ be a 
normal subgroup of $G$ such that $G/N$ is abelian. 
Then, as shown in \cite{FGS},  $C$ is of finite $q$-type when 
restricted to $\Z{N}$ if and only if $\rS(G,N)\subset \Sigma^q(C)$.

\subsection{The \texorpdfstring{$\Sigma$}{Sigma}-invariants of a CW-complex}
\label{subsec:fgs cw}
Let $X$ be a connected CW-complex with 
finite $1$-skeleton, and let $G=\pi_1(X,x_0)$ be its fundamental 
group.  A choice of classifying map $X\to K(G,1)$ yields   
an induced isomorphism, $H^1(G;\R) \isom H^1(X;\R)$, 
which identifies the respective unit spheres, 
$\rS(G)$ and $\rS(X)$.  The cell structure on $X$ lifts to a 
cell structure on the universal cover $\wX$, invariant under the 
action of $G$ by deck transformations. 
Thus, the cellular chain complex $C_*(\wX;\Z)$ 
is a chain complex of free $\Z{G}$-modules.  

\begin{definition}
\label{def:fgs cw}
For each $q> 0$, the {\em $q$-th 
Bieri--Neumann--Strebel--Renz invariant}\/ 
of $X$ is the subset of $\rS(X)$ given by
$\Sigma^q(X;\Z)=\Sigma^q(C_*(\wX;\Z))$.
\end{definition}

We will denote by $\Sigma^q(X;\Z)^{\compl}$ the complement 
of $\Sigma^q(X;\Z)$ in $\rS(X)$.  It is shown in \cite{FGS} that 
$\Sigma^q(X;\Z)$ is an open subset of $\rS(X)$, which depends 
only on the homotopy type of $X$. 

\subsection{The \texorpdfstring{$\Sigma$}{Sigma}-invariants of a group}
\label{subsec:sigma-groups} 

Let $G$ be a finitely generated group, and let $\Cay(G)$ be the Cayley 
graph associated to a fixed finite generating set. The invariant $\Sigma^1(G)$ 
of Bieri, Neumann, and Strebel \cite{BNS} is the set of homomorphisms 
$\chi\in \rS(G)$ for which the induced subgraph of $\Cay(G)$ on vertex 
set $G_{\chi}$ is connected. The BNS set is an open subset of $\rS(G)$, 
which does not depend on the choice of finite generating set for $G$. 
The rational points on $\rS(G)$ correspond to epimorphisms 
$\chi\colon G\to \Z$; the kernel of $\chi$ is 
finitely generated if and only if both $\chi$ and $-\chi$ belong to $\Sigma^1(G)$. 
Additionally, the complements of the $\Sigma$-invariants enjoy 
the following naturality property.  

\begin{proposition}[\cite{BNS}] 
\label{prop:sigma1-nat}
Suppose $\varphi\colon G\surj Q$ is a 
surjective group homomorphism. Then the induced embedding, 
$\varphi^*\colon \rS(Q)\inj \rS(G)$, restricts to an injective map 
between the complements of the respective BNS-invariants,
$\varphi^*\colon \Sigma^1(Q)^{\compl} \inj  \Sigma^1(G)^{\compl}$.
\end{proposition}

The BNS invariant was generalized by Bieri and Renz \cite{BR}, 
as follows. For a $\Z{G}$-module $M$, define  
$\Sigma^{q}(G; M)\coloneqq\Sigma^{q}(F_{\bullet})$, where 
$F_{\bullet} \to M$ is a projective $\Z{G}$-resolution of $M$.  
In particular, this yields invariants $\Sigma^{q}(G;\Z)$, 
where $\Z$ is viewed as a trivial $\Z{G}$-module; 
clearly, $\Sigma^{q}(G;\Z)=\Sigma^{q} (K(G,1),\Z)$.  
As noted in \cite[\S 1.3]{BR}, $\Sigma^1(G)=-\Sigma^1(G;\Z)$.  

\subsection{Novikov--Sikorav completion}
\label{subsec:novikov}
Let $G$ be a group. 
The {\em Novikov--Sikorav completion}\/ of the group ring $\Z{G}$ 
with respect to a homomorphism $\chi\colon G\to \R$ consists of 
all formal sums $\sum n_g g\in \Z^{G}$, 
having the property that, for each $c\in \R$, the set 
\begin{equation}
\label{eq:nov-sik}
\big\{g\in G \mid \text{$n_g \ne 0$ and $\chi(g) \ge c$}\big\}
\end{equation} 
is finite, see \cite{Novikov, Sk87, Far}. 
Equivalently, an element $\alpha \in \widehat{\Z{G}}_{\chi}$ 
can be represented as a formal sum 
$\sum_{i=1}^\infty n_i g_i$ with $n_i \in \Z \setminus \{0\}$ 
such that the sequence of 
values $(\chi(g_i))_{i \in \N}$ is non-increasing  and tends to $-\infty$., i.e., 
$\chi(g_1) \ge \chi(g_2) \ge \chi(g_3) \ge \cdots$, 
and $\lim_{i \to \infty} \chi(g_i) = -\infty$. 
With the usual addition and with multiplication defined by 
$(\sum n_g g) \cdot  (\sum m_h h) = \sum (n_g m_h) g h$, 
the Novikov--Sikorav completion, $\widehat{\Z{G}}_{\chi}$, is a ring 
containing $\Z{G}$ as a subring. Consequently, $\widehat{\Z{G}}_{\chi}$ 
carries a natural structure of left $\Z{G}$-module (we will also view 
it as a right $\widehat{\Z{G}}_{\chi}$-module). 
For instance, if 
$G=\Z=\langle t \rangle$ and $\chi(t)=1$, then 
$\widehat{\Z{G}}_{\chi}=\Z[[t^{-1}]][t]=
\{\sum_{i\le k} n_i t^i \mid \text{$n_i\in \Z$, for some $k\in \Z$}\}$. 

Moreover, the Novikov--Sikorav completion enjoys the following 
functoriality property. Let $\phi\colon G\to K$ be a homomorphism, and 
let $\bar\phi\colon \Z^{G}\to \Z^{K}$ be its linear extension to 
formal sums. If $\chi\colon K\to \R$ is a character, then 
$\bar\phi$ restricts to a morphism of topological rings 
between the corresponding completions, $\hat\phi \colon  
\widehat{\Z{G}}_{\chi\circ \phi} \to \widehat{\Z{K}}_{\chi}$. 

\subsection{Novikov--Sikorav homology}
\label{subsec:novikov-sikorav}
In his thesis \cite{Sk87}, J.-Cl.~Sikorav reinterpreted the BNS invariant 
of a finitely generated group $G$ in terms of Novikov homology. 
This interpretation was extended to all BNSR invariants 
by Bieri \cite{Bi07}, and later to the BNSR invariants of 
CW-complexes by Farber, Geoghegan, and Sch\"{u}tz \cite{FGS}. 

Let $X$ be a connected CW-complex and let $G=\pi_1(X)$. 
Recall that the homology groups of $X$ with coefficients in a left 
$\Z{G}$-module $M$ are given by 
$H_{i}(X;M)\coloneqq H_i( C_*(\widetilde{X};\Z)\otimes_{\Z{G}} M )$, 
where $\widetilde{X}$ is the universal cover of $X$ and $C_*(\widetilde{X};\Z)$ 
is its cellular chain complex, viewed as a right $\Z{G}$-module via the 
cellular action of $G$ on the chains of $\widetilde{X}$. 
We will use in an essential way the following theorem, which expresses the 
BNSR invariants of $X$ as the vanishing loci for homology with 
coefficients in the Novikov--Sikorav completions of $G$.  

\begin{theorem}[\cite{FGS}]
\label{thm:bns novikov}
If $X$ is a connected CW-complex with finite $q$-skeleton, then    
\begin{equation}
\label{eq:bnsr fgs}
\Sigma^{q}(X; \Z)= \big\{ \chi \in \rS(X) \mid
H_{i}(X; \widehat{\Z{G}}_{-\chi})=0,\: \text{for all $i\le q$} \big\}\, .
\end{equation}
\end{theorem}

In particular, the BNS set $\Sigma^1(G)=-\Sigma^1(G;\Z)$ consists 
of those characters $\chi \in \rS(G)$ for which both $H_{0}(G; \widehat{\Z{G}}_{\chi})$ 
and $H_{1}(G; \widehat{\Z{G}}_{\chi})$ vanish.

\begin{remark}
\label{rem:signs}
There were several steps along the way where we had to make a choice 
of sign in the various definitions of BNS invariants. We refer 
to \cite[Remark 5.7]{Su-mathann} for a detailed explanation.
\end{remark}

\section{An improved bound for the BNSR-invariants}
\label{sect:twist-bound}

Throughout this section, $\k$ denotes a field with a possibly 
trivial valuation $\upsilon\colon \k^\times \to \R$.
There is then an algebraically closed field $\K$ extending $\k$ 
such that the extension of $\upsilon$ to $\K$ is nontrivial 
(still denoted $\upsilon$): if $\upsilon$ is nontrivial, take 
$\K=\overline{\k}$ with an extended valuation; if $\upsilon$ 
is trivial, take $\K$ to be the field of Puiseux series 
$\bigcup_{n\ge 1}\overline{\k}(\!(t^{1/n})\!)$ when $\ch(\k)=0$, 
or $\overline{\k}(\!(t^{\Q})\!)$ otherwise; see \cite{MacStu}.

\subsection{Valuations and characters}
\label{subsec:vals-chars}
Let $G$ be a finitely generated group. Fix a representation 
$\sigma\colon G\to \GL_r(\k)$.
For a multiplicative character $\rho\colon G\to \K^{\times}$, set 
$\chi= \upsilon\circ \det \circ(\sigma\otimes\rho)$, where $\det$ 
is the determinant function. Then we obtain 
a commuting diagram,
\[
\begin{tikzcd}
    G \ar[d, "\chi"] \ar[r, "\sigma\otimes\rho"] & \GL_r(\K) \ar[d,"\det"]\\
    \R &  \K^{\times}. \ar[l, "\upsilon"]
\end{tikzcd}
\]

Let $\hat{\K}$ be the topological completion of $\K$ with respect 
to the absolute value defined by the valuation $\upsilon$.  Then $\hat{\K}$ 
is a field, and the map to the completion, $\iota\colon \K \inj \hat{\K}$, 
is a field extension.

\begin{lemma}
\label{lem:nov-mod}
Assume:
\begin{enumerate}[label=\upshape(\roman*), itemsep=2pt]
\item There exists $d\in\R$ such that, for all $g\in G$ and $1\le i,j\le r$, 
\begin{equation}
\label{eq:val-bound}
\upsilon([\sigma(g)]_{ij}) \ge d.
\end{equation}
\item $\upsilon(\det(\sigma(g)))=0$ for all $g\in G$.
\end{enumerate}
Then:
\begin{enumerate}[itemsep=2pt]
\item \label{nov1} 
$\chi = r(\upsilon\circ\rho)$;
\item \label{nov2} 
$\hat{\K}^r$ is a left $\widehat{\Z G}_{-\chi}$-module via $\sigma\otimes\rho$.
\end{enumerate}
\end{lemma}

\begin{proof}
For any $g\in G$,
\[
\chi(g)
= \upsilon(\det(\sigma(g)\otimes \rho(g)))
= \upsilon(\det(\sigma(g))) + r\,\upsilon(\rho(g))
= r\,(\upsilon\circ \rho)(g),
\]
since $\upsilon(\det(\sigma(g)))=0$. This proves \eqref{nov1}.

For \eqref{nov2}, let $\sum n_i g_i \in \widehat{\Z G}_{-\chi}$ and 
$e_k=(0,\cdots, 0, 1, 0,\cdots, 0)\in \hat{\K}^r$ the $k$-th standard basis vector. 
The $j$-th coordinate of $(\sigma\otimes\rho)(\cdot)e_k$ is
\[
\sum_i n_i\,\rho(g_i)\,[\sigma(g_i)]_{kj}.
\]
Using $\upsilon(n_i)\ge 0$, the bound on matrix entries, and \eqref{nov1}, we obtain
\[
\upsilon(n_i\rho(g_i)[\sigma(g_i)]_{kj})
\ge \tfrac{1}{r}\chi(g_i)+d.
\]
Since $\chi(g_i)\to +\infty$ in the Novikov completion, 
$\upsilon(n_i\rho(g_i)[\sigma(g_i)]_{kj}) \to +\infty$ and the terms tend to $0$
in $\hat{\K}$, proving convergence. For a general vector $e\in \hat{\K}^r$, 
each coordinate is a finite linear combination of series of the standard 
basis vectors already considered.
\end{proof}

\begin{remark}
\label{rem:val-bound}
The lower bound \eqref{eq:val-bound} is essential for the convergence
argument in Lemma~\ref{lem:nov-mod}. It ensures that the growth of
$\chi(g_i)\to +\infty$ forces the terms $n_i\,\rho(g_i)\,[\sigma(g_i)]_{kj}$ 
to tend to zero in $\hat{\K}$.

Without such a bound, the $\widehat{\Z G}_{-\chi}$-module structure on
$\hat{\K}^r$ can fail. For instance, let $G=\Z$, $\K=\Q_p$ with the $p$-adic
valuation $\upsilon_p$, and consider the representation
\[
\sigma\colon \Z\to \GL_2(\Q_p), \qquad
1\mapsto
\begin{pmatrix}
p & 0 \\
0 & p^{-1}
\end{pmatrix}.
\]
Then $\det(\sigma(n))=1$ for all $n$, but
\[
\upsilon_p(p^n)=n \to +\infty,
\qquad
\upsilon_p(p^{-n})=-n \to -\infty,
\]
so \eqref{eq:val-bound} fails. In this situation, the negative drift in the
valuations of the matrix entries cancels the positive drift coming from
$\chi(g_i)\to +\infty$, and the series defining the module action need not
converge.

Note also that this example does not satisfy any of the following conditions in
Lemma~\ref{lem:novikov-conditions}, and hence falls outside the
range of applicability of Lemma~\ref{lem:nov-mod}.
\end{remark}

\begin{lemma}
\label{lem:novikov-conditions}
Each of the following conditions ensures that the hypotheses of
Lemma~\ref{lem:nov-mod} are satisfied.
\begin{enumerate}[label=\textup{(\alph*)}]
\item \label{c1}
$\upsilon$ is trivial on $\k$.
\item \label{c2}
$\im(\sigma)$ is finite.
\item \label{c3}
$\sigma(G)\subseteq \GL_r(\cO_\k)$, where 
$\cO_\k=\{\lambda\in \k \mid \upsilon(\lambda)\geq 0\}$ is the valuation ring of $\k$.
\end{enumerate}
\end{lemma}

\begin{proof}
We check that each condition implies the two hypotheses of
Lemma~\ref{lem:nov-mod}.

\medskip
\noindent\textbf{Case \ref{c1}.}
Since $\upsilon$ is trivial on $\k$, we have
$\upsilon(\det(\sigma(g)))=0$ for all $g\in G$. 
Moreover, all nonzero matrix entries satisfy $\upsilon([\sigma(g)]_{ij})\ge 0$,
so the boundedness condition holds.

\medskip
\noindent\textbf{Case \ref{c2}.}
Since $\im(\sigma)$ is finite, all matrix entries lie in a finite subset of $\k$,
hence there exists $d\in\R$ such that $\upsilon([\sigma(g)]_{ij})\ge d$ 
for all $g,i,j$. 
Moreover, each $\sigma(g)$ has finite order, hence so does
$\det(\sigma(g))$, and therefore $\upsilon(\det(\sigma(g)))=0$.
 
\medskip
\noindent\textbf{Case \ref{c3}.}
If $\sigma(G)\subseteq \GL_r(\cO_\k)$, then for all $g$:
\[
\upsilon([\sigma(g)]_{ij}) \ge 0
\quad\text{and}\quad
\det(\sigma(g))\in \cO_\k^\times.
\]
Then  $\upsilon(\det(\sigma(g)))\geq 0$ together with 
$\upsilon(\det(\sigma(g^{-1})))\geq 0$ implies  $\upsilon(\det(\sigma(g)))= 0$.
Thus both hypotheses of Lemma~\ref{lem:nov-mod} are satisfied directly,
with $d=0$.
\end{proof}

\subsection{A key homological criterion}
\label{subsec:sigma-cjls}

Let $X$ be a connected, finite-type CW-complex, and let $G=\pi_1(X)$. 
As before, let $\rho\colon G\to \K^{\times}$ and 
$\sigma\colon G\to \GL_r(\k)$ be two group homomorphisms, 
and set $\chi= r\cdot(\upsilon\circ \rho)\colon G\to \R$.

\begin{theorem}
\label{thm:sigma-cjls}
With notation as above, assume that $\upsilon$ and $\sigma$ satisfy one of the 
conditions in Lemma~\ref{lem:novikov-conditions}.
If $\chi$ is nonzero and $H_i(X;L_{\sigma}\otimes L_\rho)\neq 0$ 
for some $0\leq i \leq q$, then $\chi \not\in \Sigma^q(X;\Z)$.
\end{theorem}

\begin{proof}
The proof is along the same lines as that of \cite[Thm.~10.1]{PS-plms} 
and \cite[Thm.~6.4]{Su-mathann}; for completeness, we give full details. 

Consider the representation $\sigma\otimes \rho\colon G\to \GL_r(\K)$ 
associated with the local system $L_{\sigma}\otimes L_{\rho}$. 
By Lemma~\ref{lem:nov-mod} and Lemma~\ref{lem:novikov-conditions}, 
the completion $\hat{\K}^r$ is a left $\widehat{\Z{G}}_{-\chi}$-module, which 
we denote by $\hat{\K}^r_{\sigma\otimes\rho}$. Restricting scalars via the 
inclusion  $\Z{G} \inj \widehat{\Z{G}}_{-\chi}$ yields the $\Z{G}$-module 
$\hat{\K}^r_{\iota\circ (\sigma\otimes\rho)}$, defined by 
the representation $\iota\circ (\sigma\otimes\rho)\colon G\to \GL_r(\hat \K)$.

For a ring $R$, a bounded below chain complex of flat right $R$-modules 
$K_{*}$, and a left $R$-module $M$, there is a (right half-plane, boundedly 
converging) K\"{u}nneth spectral sequence, 
\begin{equation}
\label{eq:base-change}
E^2_{ij}=\Tor^R_i(H_j(K),M) \Rightarrow H_{i+j}(K\otimes_R M)\, ,
\end{equation}
see \cite[Thm.~5.6.4]{Weibel}. 
We will apply this spectral sequence to the ring $R=\widehat{\Z{G}}_{-\chi}$, 
the chain complex of free $R$-modules 
$K_*= C_*\big(\widetilde{X};\Z\big) \otimes_{\Z{G}} \widehat{\Z{G}}_{-\chi}$, 
and the $R$-module $M=\hat{\K}^r_{\sigma\otimes\rho}$. 

Assume now that $\rho \in \VV^{\le q}(X,L_{\sigma})$, and suppose $\chi$ 
belongs to $\Sigma^q(X;\Z)$. By Theorem \ref{thm:bns novikov}, this condition  
implies the vanishing of $H_{j}(X; \widehat{\Z{G}}_{-\chi})$ for all $j\le q$;   
that is, $H_{j}(K_*) =0$ for $j\le q$. Therefore, $E^2_{ij}=0$ for $j\le q$. 
Noting that 
\begin{equation}
\label{eq:coefficients}
K_*\otimes_R M= C_*\big(\widetilde{X};\Z\big) \otimes_{\Z{G}} 
\widehat{\Z{G}}_{-\chi} \otimes_{\widehat{\Z{G}}_{-\chi}} \hat{\K}^r_{\sigma\otimes\rho}
=  C_*\big(\widetilde{X};\Z\big) \otimes_{\Z{G}} \hat\K^r_{\iota\circ (\sigma\otimes\rho)}\, ,
\end{equation}
we infer from \eqref{eq:base-change}  that 
$H_{i+j}\big(X;\hat{\K}^r_{\iota\circ (\sigma\otimes\rho)}\big)=0$ for $j\le q$, and so 
$H_{j}\big(X;\hat{\K}^r_{\iota\circ (\sigma\otimes\rho)}\big)=0$ for $j\le q$. From the 
definition of the characteristic varieties (over the field $\hat{\K}$), this is equivalent to   
$\iota\circ \rho \notin \VV^{\le q}(X,L_{\sigma}\otimes_{\K}\hat\K)$. 
It follows that 
$\rho \notin \VV^{\le q}(X,L_{\sigma})$, which contradicts our hypothesis on $\rho$. 
Therefore, $\chi\notin \Sigma^q(X;\Z)$, and we are done.
\end{proof}

\subsection{Main results}
\label{subsec:main}

We are now in a position to prove Theorem~\ref{thm:bns-twisted-trop-intro}
from the Introduction. The key input is the Novikov module structure
established in Lemma~\ref{lem:nov-mod}, together with the valuation
conditions ensured by Lemma~\ref{lem:novikov-conditions}.

\begin{theorem}
\label{thm:bns-twisted-trop}
Let $X$ be a connected finite CW-complex with $\pi_1(X)=G$. 
Fix a field $\k$ with a rational valuation
$\upsilon$ (i.e., $\upsilon(\k^{\times})\subseteq \Q$), possibly trivial. 
Let $\sigma\colon G\to \GL_r(\k)$ be a representation 
and let $L_\sigma$ denote the corresponding local system on $X$. 
Consider the twisted jump loci 
$\VV^{\le q}(X,L_\sigma)\subset H^1(X;\k^{\times})$ 
and let $\Trop_{\k,\upsilon}( \VV^{\le q}(X,L_\sigma))
\subset H^1(X;\R)$ be its tropicalization with respect 
to the valuation $\upsilon$. Assume that $\upsilon$ and $\sigma$ satisfy one of the 
valuation conditions from 
Lemma~\ref{lem:novikov-conditions}.  Then 
\[
\Sigma^q(X;\Z) \subseteq 
\rS\big(\!\Trop_{\k,\upsilon}(\VV^{\le q}(X,L_{\sigma}))\big)^{\compl}.
\]
\end{theorem}

\begin{proof}
By the assumption $\upsilon(\k^\times)\subseteq \Q$, we can assume that 
$\upsilon(\K^\times)\subseteq \Q$. Consider a $\K$-valued, multiplicative 
character $\rho\colon \pi_1(X)\to \K^{\times}$. Composing the valuation map 
$\upsilon\colon \K^{\times}\to \Q$ with the homomorphism 
$\rho$, we obtain an additive character, 
$\chi= r\cdot(\upsilon\circ \rho)\colon \pi_1(X) \to \Q$, 
which may be viewed as an element of $H^1(X;\Q)$, that is, 
a rational point in $H^1(X;\R)$. 
Moreover, if $\chi$ is nonzero, then $\chi$ 
determines a rational point on $\rS(X)$. 
Since we work after applying $\rS(-)$, multiplication by the positive 
scalar $r$ does not change the corresponding point of the character sphere.

By the assumption $\upsilon(\k^{\times})\subseteq \Q$ and 
\cite[Thm.~7.11]{Ra12}, any nonzero rational point  
$\chi \in \Trop_{\k,\upsilon}(\VV^{\le q}(X,L_\sigma))$ 
is of the form $r\cdot(\upsilon\circ \rho)$, 
for some $\rho\in \VV^{\le q}(X, L_\sigma)\times_{\k}\K$.
Since  $\upsilon$ and $\sigma$ satisfy one of the 
valuation conditions from 
Lemma~\ref{lem:novikov-conditions}, 
Theorem~\ref{thm:sigma-cjls} implies that $\chi$ 
belongs to $\Sigma^q(X;\Z)^{\compl}$.  
Note that by the assumption that $\upsilon(\k^{\times})\subseteq \Q$ 
and Theorem~\ref{thm:structure-trop}, 
the set of rational points are dense in 
$\rS(\Trop_{\k,\upsilon}(  \VV^{\le q}(X,L_\sigma)))$. Moreover, 
$\Sigma^q(X;\Z)^{\compl}$ is a closed subset of $\rS(X)$.  Thus, 
$\rS(\Trop_{\k,\upsilon}( \VV^{\le q}(X,L_\sigma)))
\subseteq \Sigma^q(X;\Z)^{\compl}$. 
\end{proof}

\begin{corollary}
\label{cor:twist-bound}
Let $X$ be a connected finite CW-complex with $\pi_1(X)=G$. 
For any representation $\sigma\colon G\to \GL_r(\Z)$ with 
$\cJ_{\leq q}(X,L_\sigma)\coloneqq \bigcap_{0\leq i \leq q} \cJ_{ i}(X,L_\sigma)$, 
we have
\[
\Sigma^q(X;\Z) \subseteq 
\rS\big(\!\Trop_\Z( \cJ_{\le q}(X,L_{\sigma}))\big)^{\compl}.
\]
\end{corollary}

\begin{proof}
Set $I=\cJ_{\le q}(X,L_{\sigma})$.    
In each of the three coefficient/valuation settings below considered 
in Proposition~\ref{prop:three-trop-union-Z}, at least one
of the conditions in Lemma~\ref{lem:novikov-conditions} holds:
\begin{itemize}[itemsep=1.5pt]
\item $I\otimes \Q = \cJ_{\le q}(X,L_{\sigma}\otimes \Q)$ with $\upsilon_0$;
\item $I\otimes \Q = \cJ_{\le q}(X,L_{\sigma}\otimes \Q)$ with $\upsilon_p$;
\item $I\otimes \F_p = \cJ_{\le q}(X,L_{\sigma}\otimes \F_p)$ with $\hat{\upsilon}_p$.
\end{itemize}    
    
For $\upsilon_0$ and $\hat{\upsilon}_p$, condition~\ref{c1} holds.

For $\upsilon_p$, condition~\ref{c3} holds since
$\sigma(G)\subseteq \GL_r(\Z)$ implies
$\upsilon_p([\sigma(g)]_{ij})\ge 0$ and
$\det(\sigma(g))=\pm 1$, hence $\upsilon_p(\det(\sigma(g)))=0$.

By Lemma~\ref{lem:novikov-conditions}, the hypotheses of
Lemma~\ref{lem:nov-mod} are satisfied in each case.
Hence Theorem~\ref{thm:bns-twisted-trop} applies, giving
\[
\rS(\Trop_{\k,\upsilon}( I\otimes\k))
\subseteq \Sigma^q(X;\Z)^{\compl}
\]
for each pair $(\k,\upsilon)$ above.
The result follows from Proposition~\ref{prop:three-trop-union-Z}.
\end{proof}

\begin{remark}
\label{rem:integral-vs-finite}
The assumption $\im(\sigma)\subset \GL_r(\Z)$ in 
Corollary~\ref{cor:twist-bound} is weaker than 
requiring $\im(\sigma)$ to be finite: the integrality of 
the matrix entries ensures the valuation conditions from 
Lemma~\ref{lem:novikov-conditions} for all three types of valuations 
simultaneously. However, in the applications to 
$3$-manifold groups (Theorem~\ref{thm:sigma-3mfd}), 
the representations produced by the Friedl--Vidussi theorem 
do have finite image. Whether infinite-image integral representations can 
yield sharper upper bounds for $\Sigma^q(G;\Z)$ than 
finite-image ones is an open problem directly 
relevant to Question~\ref{que:sigma}.
\end{remark}

\subsection{An example}
\label{subsec:ex-bns-twistjump}
We give an example where one  can identify the sets $\Sigma^1(X;\Z)$ by 
means of the twisted jump loci, but not the ordinary jump loci.

\begin{example}
\label{ex:one-rel}
Consider the one-relator group $G=\langle x_1,x_2 \mid R\rangle$, where 
\[
R= x_1^{-1} x_2^{-1} x_1 x_2^2 x_1^{-1} x_2^{-1}
x_1^2 x_2^{-1} x_1^{-1} x_2 x_ 1^{-1} x_2 x_1 x_2^{-1} .
\]
As an application of his algorithm, Brown \cite{Br} showed that 
$\Sigma^1(G)$ consists of two open arcs on the unit circle $S^1=\rS(G)$, 
joining the points $(-1,0)$ to $(0,-1)$ and $ (0,-1)$ to 
$(\sfrac{1}{\sqrt{2}},\sfrac{1}{\sqrt{2}})$, 
respectively; see Figure~\ref{fig:-BNS}. 

Computing Fox derivatives gives 
\begin{align*}
\frac{\partial{R}}{\partial x_1}
&=-x_1^{-1}(1 -x_2^{-1})- [x_1^{-1},x_2^{-1}] x_2 x_1^{-1}(1-x_2^{-1}) + 
[x_1^{-1},x_2^{-1}][x_2,x_1^{-1}](1-x_1x_2^{-1}x_1^{-1})-\\
&\qquad 
[x_1^{-1},x_2^{-1}][x_2,x_1^{-1}][x_1,x_2^{-1}]x_1^{-1}(1-x_2),
\\
\frac{\partial{R}}{\partial x_2}
&=-x_1^{-1}x_2^{-1} (1-x_1) + [x_1^{-1},x_2^{-1}](1-x_2 x_1^{-1}x_2^{-1}) - 
[x_1^{-1},x_2^{-1}][x_2,x_1^{-1}]x_1 x_2^{-1}(1-x_1^{-1}) + \\
&\qquad 
[x_1^{-1},x_2^{-1}][x_2,x_1^{-1}][x_1,x_2^{-1}]x_1^{-1}(1-x_2 x_1 x_2^{-1}).
\end{align*}
Applying the abelianization map $\ab\colon G\to \Z^2$, $x_i\mapsto t_i$ to 
these Fox derivatives yields the matrix
\[
\partial_2^{\ab}= 
 t_1^{-1}t_2^{-1}(t_1-1)\begin{pmatrix}
 t_2-1  & 1-t_1
\end{pmatrix},
\]
from which we get   
$\sqrt{\cJ_{\leqslant 1}(G; \Z)}=(t_1-1)$ and 
$\Trop_\Z( \cJ_{\leqslant 1}(G; \Z))=\{0\}\times \R$.

Consider now the group epimorphism 
\[
\pi\colon G\longsurj S_3=\langle x_1,x_2\mid 
x_1^3=x_2^2=(x_2x_1)^2=1\rangle
\]
and the $2$-dimensional representation $S_3\to \GL_2(\Z)$ given by 
\begin{equation*}
 x_1 \mapsto  \begin{pmatrix}
      -1 & 1 \\
     -1 & 0
 \end{pmatrix} , \quad 
 x_2\mapsto \begin{pmatrix}
     0 & 1 \\
     1 & 0
 \end{pmatrix}.
\end{equation*}
Letting $\sigma\colon G\to \GL_2(\Z)$ be the resulting representation and 
proceeding as before, we find that the map $\partial_2^{\sigma}$ has matrix  
\begin{equation*}
\frac{1}{t_1t_2}
\begin{pmatrix}
(1 + t_2) (-1 + t_1 + 2 t_2)& -t_2 (1 + t_2)& 
1 - t_1^2 - t_2 - 2 t_1 t_2& -t_1 + t_1^2 + 2 t_2 + t_1 t_2
\\
-1 + t_1 + t_2 + t_2^2& 1 - t_1 + t_1 t_2 - 2 t_2^2&
  1 - t_1 - 2 t_2 - t_1 t_2 & (-1 + t_1) (1 + t_1 - t_2)
\end{pmatrix}.
\end{equation*}
Hence, the GCD of the $2\times 2$ minors of this matrix (up to units 
in $\Z[t_1^{\pm 1}, t_2^{\pm 1}]$), is  $ (1 - t_1)^2 - 3 t_2^2$. Hence 
$((1 - t_1)^2 - 3 t_2^2)\subseteq \cJ_{\leq 1}(G,L_\sigma) $

See Figure~\ref{fig:total} below for the various tropicalizations of the principal ideal
$I=((1 - t_1)^2 - 3 t_2^2)$, where Figure~\ref{fig:z-trop} is 
$\Trop_\Z(I)$ and  Figures \ref{fig:trivial-valuation-bis}, 
\ref{fig:mod3-valuation}, and \ref{fig:3adic-valuation} 
are the tropical varieties of $I$ 
considered in $ (\Q[x_1^{\pm 1},x_2^{\pm 1}],\upsilon_0)$, 
$(\mathbb{F}_3[x_1^{\pm 1},x_2^{\pm 1}],\hat{\upsilon}_3)$ and 
$(\Q[x_1^{\pm 1},x_2^{\pm 1}], \upsilon_3)$, respectively. Note that the union of 
Figures \ref{fig:trivial-valuation-bis}, \ref{fig:mod3-valuation}, and 
\ref{fig:3adic-valuation} is different from Figure~\ref{fig:z-trop}, 
while the projections onto the unit sphere are the same. Moreover, we have 
\[
-\Sigma^1(G)=\rS\big(\!\Trop_\Z( I)\big)^{\compl}=
\rS\big(\!\Trop_\Z( \cJ_{\le 1}(X,L_\sigma))\big)^{\compl}.
\]
By contrast, with $\sigma = \id$ (untwisted), one has 
$\Trop_\Z(\cJ_{\le 1}(G;\Z)) = \{0\}\times \R$, 
which gives only the two-point complement $\{(0,\pm 1)\}$ 
on the unit circle---a strictly weaker bound. The twisted variety recovers 
$-\Sigma^1(G)$ exactly.
\end{example}

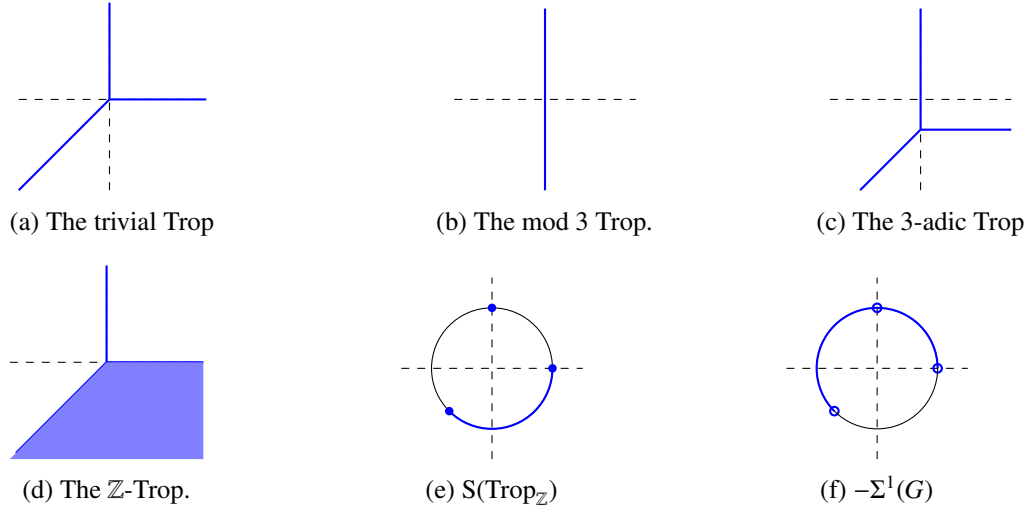
\begin{figure}[ht!]
    \centering
    \begin{subfigure}[b]{0.33\textwidth}
        \centering
        \begin{tikzpicture}[scale=0.4]
            \draw[-,dashed] (-3,0) -- (3,0);  % x-axis
            \draw[-,dashed] (0,-3) -- (0,3);  % y-axis
            \draw[-,blue,thick] (0,0) -- (0,3.2);  % Vertical blue line
            \draw[-,blue,thick] (0,0) -- (-3,-3);  % Diagonal blue line (negative slope)
            \draw[-,blue,thick] (0,0) -- (3.2,0);  % Horizontal blue line
        \end{tikzpicture}
        \caption{The trivial Trop}
        \label{fig:trivial-valuation-bis}
    \end{subfigure}
    \hfill
    \begin{subfigure}[b]{0.33\textwidth}
        \centering
        \begin{tikzpicture}[scale=0.4]
            \draw[-,dashed] (-3,0) -- (3,0);  % x-axis
            \draw[-,dashed] (0,-3) -- (0,3);  % y-axis
            \draw[-,blue,thick] (0,3) -- (0,-3);  % Diagonal yellow line
        \end{tikzpicture}
        \caption{The mod $3$ Trop.}
        \label{fig:mod3-valuation}
    \end{subfigure}
    \hfill
    \begin{subfigure}[b]{0.23\textwidth}
        \centering
        \begin{tikzpicture}[scale=0.4]
            \draw[-,dashed] (-3,0) -- (3,0);  % x-axis
            \draw[-,dashed] (0,-3) -- (0,3);  % y-axis
            \draw[-,blue,thick] (0,-1) -- (-2,-3);  % Diagonal blue line
            \draw[-,blue,thick] (0,-1) -- (0,3);  % Vertical blue line
            \draw[-,blue,thick] (0,-1) -- (3,-1);  % Horizontal blue line
        \end{tikzpicture}
        \caption{The $3$-adic Trop}
        \label{fig:3adic-valuation}
    \end{subfigure}
     \vspace{0.3cm} % Space between the rows

    \begin{subfigure}[b]{0.32\textwidth}
        \centering
        \begin{tikzpicture}[scale=0.4]
            \draw[-,dashed] (-3.2,0) -- (3,0);  % x-axis
            \draw[-,dashed] (0,-3.2) -- (0,3);  % y-axis
            \draw[-,blue,thick] (0,0) -- (0,3.2);  % Vertical blue line
            \draw[-,blue,thick] (0,0) -- (-3,-3);  % Diagonal blue line (negative slope)
            \draw[-,blue,thick] (0,0) -- (3.2,0);  % Horizontal blue line
            \fill[blue!50] (0,0) -- (3.2,0) -- (3.2,-3.2)-- (0,-3.2) -- (-3.2,-3.2) -- cycle;
        \end{tikzpicture}
        \caption{The $\Z$-Trop.}
        \label{fig:z-trop}
    \end{subfigure}
    \hfill
     \begin{subfigure}[b]{0.32\textwidth}
        \centering
        \begin{tikzpicture}[scale=0.4]
            \draw[-,dashed] (-3,0) -- (3,0);
            \draw[-,dashed] (0,-3) -- (0,3);
            \draw (0,0) circle (2);
            \fill[blue] (0,2) circle (4pt);
            \fill[blue] (2,0) circle (4pt);
            \fill[blue] (-1.414213562, -1.414213562) circle (4pt);
            \draw[blue, thick] 
                (-1.414213562, -1.414213562) arc[start angle=225, end angle=360, radius=2];
        \end{tikzpicture}
        \caption{$\rS(\Trop_\Z)$}
        \label{fig:example2_BNS-bis}
    \end{subfigure}
    \hfill
     \begin{subfigure}[b]{0.32\textwidth}
        \centering
        \begin{tikzpicture}[scale=0.4]
            \draw[-,dashed] (-3,0) -- (3,0);
            \draw[-,dashed] (0,-3) -- (0,3);
            \draw[blue,thick] (0,2) circle (4pt);
            \draw[blue,thick] (2,0) circle (4pt);
            \draw[blue,thick] (-1.414213562, -1.414213562) circle (4pt);
            \draw (0,0) circle (2);
            \draw[blue, thick] 
                (2,0) arc[start angle=0, end angle=225, radius=2];
        \end{tikzpicture}
        \caption{$-\Sigma^1(G)$}
        \label{fig:-BNS}
    \end{subfigure}
\caption{Tropical varieties and $\Sigma$-invariant of the 
group $G$ from Example~\ref{ex:one-rel}}
\label{fig:total}
\end{figure}

\section{K\"{a}hler groups}
\label{sect:Kahler}

In this section we apply the general results of the 
preceding sections to compact K\"{a}hler manifolds and 
their fundamental groups. A finitely presented group $G$ 
is called a \emph{K\"{a}hler group} if it is isomorphic 
to $\pi_1(X)$ for some compact K\"{a}hler manifold $X$.

The key feature of K\"{a}hler groups that we exploit 
is that both their first characteristic variety $\VV^1(X,\C)$ 
and their BNS invariant $\Sigma^1(G)$ are completely 
determined by the \emph{hyperbolic orbifold fibrations} 
of $X$ — that is, surjective holomorphic maps from $X$ 
onto orbifold curves with negative orbifold Euler 
characteristic. These fibrations are finite in number 
up to equivalence \cite[Thm.~2]{De08}, and their 
cohomological footprints tile both invariants, with 
one subtle difference: certain genus-$1$ fibrations 
with some multiple fibers contribute to 
$\Sigma^1(G)^{\compl}$ but are invisible in $\VV^1(X,\C)$.

In \S\ref{subsec:orbifolds} we set up the algebraic 
foundations, computing the characteristic varieties 
of compact orbifold groups $\pi_1^{\orb}(C_g,\bm{\mu})$ 
via Fox calculus. This includes both the twisted version 
needed for Alexander polynomial computations and the 
untwisted formula \eqref{eq:v1piorb}, which feeds into 
the structure theorem for $\VV^1(X,\C)$.
In \S\ref{subsc:cv-bns-kahler} we state and compare 
the two main structure theorems: Theorem~\ref{thm:cv1-kahler} 
for $\VV^1(X,\C)$, due to Beauville, Arapura, Campana and 
others, and Theorem~\ref{thm Delzant} for $\Sigma^1(G)$, 
due to Delzant. We also derive a corollary 
(Corollary~\ref{cor:bns-kahler}) relating the tropicalization 
of twisted characteristic varieties to the BNS invariant. 
In \S\ref{subsec:twist-alex-kahler} we prove 
Theorem~\ref{thm:alexpoly-kahler}, one of the main results 
of the paper: the first twisted Alexander polynomial 
$\Delta^\sigma(X)$ of a compact K\"{a}hler manifold is 
always $0$ or $1$, providing a new obstruction to 
K\"{a}hlerianity that goes beyond the untwisted DPS 
obstruction \cite{DPS-imrn}. In \S\ref{subsec:examples}, we give some 
examples where Theorem~\ref{thm:intro-kahler} detects non-K\"ahlerianity
and the untwisted result does not.
\subsection{Compact orbifold groups}
\label{subsec:orbifolds}
Let $C_{g}$ be a compact Riemann surface of genus 
$g\geqslant 1$ and let $s\geqslant 0$ be an integer. If $s>0$, fix points 
$\{q_1,\ldots,q_s\}$ in $C_g$ and  assign to these points an integer weight 
vector ${\bm \mu}=(\mu_1,\cdots,\mu_s)$ with all $\mu_i\geqslant 2$.
The orbifold fundamental group $\Gamma=\pi_1^{\orb}(C_{g},{\bm \mu})$ 
associated to these data has presentation  
\[ 
\Gamma=\Bigl\langle x_{1},\ldots,x_{g},y_{1},\ldots,y_{g},z_{1},\ldots,z_{s} \bigm| \prod_{i=1}^{g}
[x_{i},y_{i}]\cdot\prod_{j=1}^{s} z_{j}=1, \ z_{j}^{\mu_{j}}=1 \text{ for all } 1\leqslant 
j\leqslant s   \Bigr\rangle. 
\]
Next we compute the first twisted Alexander polynomial of a compact orbifold group.

\begin{proposition}
\label{prop:compact-orbifold-group}
Let $\Gamma=\pi_1^{\orb}(C_g,\bm{\mu})$, and fix an algebraically
closed field $\k$. Let $\sigma \colon \Gamma \to \GL_r(\k)$ 
be a representation, sending generators $x_i,y_i,z_j$ to matrices
$X_i,Y_i,Z_j$, respectively, such that
\[
\prod_{i=1}^{g}[X_i,Y_i]\cdot \prod_{j=1}^{s} Z_j = I_r,
\qquad
Z_j^{\mu_j}=I_r \quad (1\le j\le s).
\]

Then the following hold:
\begin{enumerate}
\item If $g\ge 2$, or if the matrix $1+Z_j+\cdots+Z_j^{\mu_j-1}$ 
is not invertible for some $j$, then $\Delta^\sigma(\Gamma)=0$.

\item If $g=1$ and the matrices $1+Z_j+\cdots+Z_j^{\mu_j-1}$ 
are invertible for all $1\le j\le s$, then $\Delta^\sigma(\Gamma)=1$.
\end{enumerate}

In particular, the invertibility condition holds whenever
$Z_j=I_r$ and $\ch(\k)\nmid \mu_j$.
\end{proposition}    

\begin{proof}
Consider the compact orbifold group $\Gamma=\pi_1^{\orb}(C_{g},{\bm \mu})$ 
as above, and denote by $X=K(\Gamma,1)$ its classifying space. By Fox calculus, 
the boundary map $\partial_2$ in $C_*(\widetilde{X};\k)$ is given by the 
$(2g+s)\times (s+1)$ matrix 
\[ 
\begin{pmatrix} 
\label{matrix compact} 
\frac{z_{1}^{\mu_{1}}-1}{z_{1}-1} & 0 & \cdots & 0 & (\prod_{j=1}^s z_j)^{-1} \\
0 & \frac{z_{2}^{\mu_{2}}-1}{z_{2}-1} & \cdots & 0 & (\prod_{j=2}^s z_j)^{-1} \\
\vdots & \vdots & \vdots & \vdots & \vdots \\
0 & 0 & \cdots & \frac{z_{s}^{\mu_{s}}-1}{z_{s}-1} &  z_s^{-1} \\
0& 0 & \cdots & 0 & 1-x_1y_1x_1^{-1} \\
0& 0 & \cdots & 0 & x_1-[x_{1},y_{1}] \\
0& 0 & \cdots & 0 & [x_1,y_1] (1-x_2y_2x_2^{-1}) \\
0& 0 & \cdots & 0 & [x_1,y_1] (x_2-[x_{2},y_{2}]) \\
\vdots& \vdots & \vdots & \vdots & \vdots \\
0& 0 & \cdots & 0 & \prod_{k=1}^{g-1}[x_k,y_k] (1-x_g y_g x_g^{-1}) \\
0& 0 & \cdots & 0 & \prod_{k=1}^{g-1}[x_k,y_k] (x_g-[x_{g},y_{g}]) \\
\end{pmatrix}. 
\]
Here we adopt the notation $\frac{z_{j}^{\mu_{j}}-1}{z_{j}-1}\coloneqq 
1+z_j+\cdots+z_j^{\mu_j-1}$.

Consider a character $\rho\in \Hom^0(\Gamma, \k^{\times})$, 
the connected component of $\Hom(\Gamma,\k^{\times})$ 
containing the constant sheaf, which sends $x_i\to \rho_i$, 
$y_i\to \rho_{g+i}$, and $z_j\to 1$. 
Then $\partial_2^{\sigma \otimes \rho} $ has the following 
$(2g+s)\times (s+1)$ matrix:
\[ 
\begin{pmatrix} 
\label{matrix compact bis} 
\frac{ Z_{1}^{\mu_{1}}-I_r}{ Z_{1}-I_r} & 0 & \cdots & 0 & (\prod_{j=1}^s  Z_j)^{-1} \\
0 & \frac{ Z_{2}^{\mu_{2}}-I_r}{ Z_{2}-I_r} & \cdots & 0 & (\prod_{j=2}^s  Z_j)^{-1} \\
\vdots & \vdots & \vdots & \vdots & \vdots \\
0 & 0 & \cdots & \frac{Z_{s}^{\mu_{s}}-I_r}{ Z_{s}-I_r} & Z_s^{-1} \\
0& 0 & \cdots & 0 & I_r-\rho_{g+1}\cdot X_1Y_1X_1^{-1} \\
0& 0 & \cdots & 0 & \rho_1\cdot X_1-[X_{1},Y_{1}] \\
0& 0 & \cdots & 0 & [X_1,Y_1] (I_r-\rho_{g+2}\cdot X_2Y_2X_2^{-1}) \\
0& 0 & \cdots & 0 & [X_1,Y_1] (\rho_2\cdot X_2-[X_{2},Y_{2}]) \\
\vdots& \vdots & \vdots & \vdots & \vdots \\
0& 0 & \cdots & 0 & \prod_{k=1}^{g-1}[X_k,Y_k] (I_r-\rho_{2g}\cdot X_g Y_g X_g^{-1}) \\
0& 0 & \cdots & 0 & \prod_{k=1}^{g-1}[X_k,Y_k] (\rho_g\cdot X_g-[X_{g},Y_{g}])
\end{pmatrix},
\]
where $\frac{Z_{j}^{\mu_{j}}-I_r}{Z_{j}-I_r}\coloneqq 1+Z_j+\cdots+Z_j^{\mu_j-1}$.

Assume that for some $1\leq i\leq g$, one of the following holds:
\begin{itemize}[itemsep=2pt]
\item $\rho_i$ is not an eigenvalue of $[X_i,Y_i]X_i^{-1}$ and $X_i^{-1}$; or 
\item   $\rho_{g+i}^{-1}$ is not an eigenvalue 
of $X_iY_iX_i^{-1}$ and $Y_i$.
\end{itemize}
Then
\[
\rank \partial_2^{\sigma \otimes \rho} = r+ \sum_{j=1}^s 
\rank \tfrac{Z_{j}^{\mu_{j}}-I_r}{ Z_{j}-I_r}
\]
and $\rank \partial_1^{\sigma \otimes \rho} =r$. 
 Putting all together, we get that 
\begin{align*}
    \dim H_1(\Gamma; L_\sigma\otimes L_\rho)& 
    =r(2g+s)-\rank \partial_1^{\sigma \otimes \rho} 
    -\rank \partial_2^{\sigma \otimes \rho}\\
    &=r(2g-2+s)-\sum_{j=1}^s  \rank  \tfrac{Z_{j}^{\mu_{j}}-I_r}{ Z_{j}-I_r}.
\end{align*} 
 Hence if either $g\geq 2 $ or $ \frac{Z_{j}^{\mu_{j}}-I_r}{Z_{j}-I_r}$ is 
 not of full rank for some $j$, then 
 \[
 \VV_1(\Gamma,L_\sigma)\cap \Hom^0(\Gamma,\k^{\times})= 
 \Hom^0(\Gamma,\k^{\times}).
\]
In particular, the first Alexander polynomial $\Delta^\sigma(\Gamma)=0$ in this case.
On the other hand, if $g=1$ and $ \frac{Z_{j}^{\mu_{j}}-I_r}{ Z_{j}-I_r}$ 
has full rank  for all $1\leq j \leq s$, then 
$\VV_1(\Gamma,L_\sigma)\cap \Hom^0(\Gamma,\k^{\times})$ consists 
of only finitely many points. In particular, the first Alexander polynomial 
$\Delta^\sigma(\Gamma)=1$ in this case.
\end{proof}

\begin{remark}
\label{rem:v1-gamma}
As a special case of the above computation (taking $r=1$ 
and $\sigma = \id$), we obtain an explicit description of 
the first characteristic variety of the compact orbifold group. 
Identify $\widehat{\Gamma} = \widehat{\Gamma}^{0} \times \widehat{A}$, 
where $\widehat{\Gamma}^{\circ} \cong (\C^{\times})^{2g}$ 
is the identity component and $A = \Tors(\Gamma_{\ab})$ 
is isomorphic to $\Z_{\mu_1}\oplus \cdots \oplus \Z_{\mu_s}/\langle(1,\dots ,1)\rangle$. 
Then
\begin{equation}
\label{eq:v1piorb}
\VV^1(\Gamma,\C)=\begin{cases}
\widehat{\Gamma} & \text{if $g\geq 2$}, \\[2pt]
\big(\widehat{\Gamma}\setminus\widehat{\Gamma}^{0}\big) \cup \{\bo\} 
& \text{if $g=1$ and $\abs{A}=\dfrac{\prod_{j=1}^s \mu_j}{\lcm(\mu_1,\ldots,\mu_s)} >1$}, \\[2pt]
\{\bo\} & \text{otherwise},
\end{cases}
\end{equation}
see \cite[Prop.~3.11]{ACM} and \cite[\S 9.1]{Su-imrn}.  
In particular, $\VV^1(\Gamma,\C)$ has a positive-dimensional 
component through $1$ if and only if $g \geq 2$, and has 
torsion-translated positive-dimensional components if and 
only if $\abs{A} > 1$. 
\end{remark}

\subsection{Characteristic varieties and BNS invariants 
of K\"{a}hler manifolds}
\label{subsc:cv-bns-kahler}

Let $X$ be a compact K\"{a}hler manifold; its fundamental 
group $G = \pi_1(X)$ is called a \emph{K\"{a}hler group}. 
The structure of both $\VV^1(X,\C)$ and $\Sigma^1(G)$ is 
controlled by the orbifold fibrations of $X$, as described 
in the two theorems below.

\begin{definition}
\label{def:orbi-fibration}
Let $X$ be a compact K\"ahler manifold with $\pi_1(X)=G$. A holomorphic map 
$f\colon X \to C_{g}$ is called an orbifold fibration if $f$ is surjective with connected 
fibers onto a Riemann surface $C_g$ with genus $g\geqslant 1$. Assume that $f$ 
has multiple fibers over the points  $\{q_1,\ldots,q_s\}$ in $C_g$ and let $\mu_j$ 
denote the multiplicity of the multiple fiber $f^{*}(q_j)$ (the $\gcd$ of the coefficients 
of the divisor $f^* q_j$). 
Such an orbifold fibration is denoted by $f\colon X \to (C_{g},{\bm \mu})$, 
where ${\bm \mu}=(\mu_1,\ldots,\mu_s)$. It is called {\em hyperbolic} if 
the orbifold Euler characteristic of the base curve, 
\[
\chi(C_{g},{\bm \mu}) \coloneqq 2-2g- \sum_{j=1}^s \Bigl(1-\frac{1}{\mu_j}\Bigr)
\]
is strictly negative.
\end{definition}

Two orbifold fibrations $f\colon X \to (C_{g},{\bm \mu})$ and 
$f'\colon X \to (C_{g'},{\bm \mu}')$  are equivalent
if there is a biholomorphic map $h\colon C_g \to C_{g'}$ which 
sends marked points to marked
points while preserving multiplicities. A compact K\"ahler
manifold $X$ admits only finitely many equivalence classes 
of hyperbolic orbifold fibrations, see e.g.~\cite[Thm.~2]{De08}.

The structure of $\VV^1(X,\C)$ for a compact K\"{a}hler 
manifold $X$ is described by the following result, which 
predates and is used in the proof of Theorem~\ref{thm Delzant}.

\begin{theorem}[\cite{Be92, Arapura, Cm01, Di07, De08, Cm11, ACM,DPS-duke, GL91}]
\label{thm:cv1-kahler}
Let $X$ be a compact K\"{a}hler manifold. Then
\begin{equation}
\label{eq:cv1-kahler}
\VV^1(X,\C) = \bigcup_{f} f^\#\big(\VV^1(\pi_1^{\orb}
(C_{g_f}, \bm{\mu}_f),\C)\big) \cup Z,
\end{equation}
where $Z$ is a finite set of torsion characters, and the 
union runs over the finite set of equivalence classes of 
orbifold fibrations $f\colon X\to (C_{g_f},\bm{\mu}_f)$ 
satisfying either $g_f\geq 2$, or $g_f=1$ and $\abs{A_f} >1$ 
(where $A_f$ denotes the torsion part of the abelianization 
of $\pi_1^{\orb}(C_{g_f}, \bm{\mu}_f)$).
\end{theorem}

Using Simpson's results \cite[Thm.~1]{Sim}, Delzant gave a complete 
description of $\Sigma^1(G)$ for $G$ a K\"ahler group in \cite[Thm.~1.1]{De10}. 

\begin{theorem}[{\cite[Thm.~1.1]{De10}}] 
\label{thm Delzant} 
Let $X$ be a compact K\"ahler manifold with $\pi_1(X)=G$. Then we have
\[
\Sigma^1(G;\Z) = \rS \bigg(\!\bigcup_f \im(f^{*}\colon 
H^1(C_g;\R)\longrightarrow  H^1(X;\R) ) \bigg)^{\compl},
\]
where the union runs over all equivalence classes of 
hyperbolic orbifold fibrations 
$f\colon X\to C_g$. In particular, it is a finite union. 
Moreover, $\Sigma^1(G;\Z)=\emptyset$ if and only if there exists a 
hyperbolic orbifold fibration $f\colon X \to C_g$ such that 
$f^*\colon H^1(C_g;\R)\to H^1(X;\R)$ is an isomorphism.
\end{theorem}

\begin{remark}
\label{rem:cv1-vs-sigma1}
Comparing Theorem~\ref{thm:cv1-kahler} with 
Theorem~\ref{thm Delzant}, one sees that the two 
formulas involve the same orbifold fibrations, with 
one exception: a hyperbolic orbifold fibration 
$f\colon X\to (C_1,\bm{\mu})$ with $g_f=1$ and 
$\abs{A_f} = 1$  contributes to $\Sigma^1(\pi_1(X))^{\compl}$ 
via Theorem~\ref{thm Delzant}, 
but does \emph{not} contribute a positive-dimensional 
component to $\VV^1(X,\C)$.

The reason is visible from formula~\eqref{eq:v1piorb}: when $g=1$ 
and $\abs{A}=1$, we have  $\VV^1(\pi_1^{\orb}(C_1,\bm{\mu}),\C) = \{\bo\}$, 
and the pullback $(f)^\#(\VV^1(\pi_1^{\orb}(C_1,\bm{\mu}),\C))$ 
contributes nothing positive-dimensional to $\VV^1(X,\C)$.

However, as shown by Liu--Liu \cite{LL}, the discrepancy 
disappears when one considers characteristic varieties 
over a field $\k$ of characteristic $p>0$ with $p\mid \mu_j$ for some $j$: 
in that case by the computation in the proof of Proposition \ref{prop:compact-orbifold-group}
the fibration $f$ does contribute a positive-dimensional 
component to $\VV^1(X;\k)$. 
This shows that the union over all fields $\k$ 
(including those of characteristic $p>0$) 
of the tropicalizations $\Trop_\k(\VV^1(X;\k))$ 
recovers the full $\Sigma^1(G)^{\compl}$, 
consistently with the formula of \cite{LL}.
\end{remark}

Based on Delzant's theorem and Theorem \ref{thm:bns-twisted-trop}, 
we obtain the following corollary, which strengthens results 
from \cite{Su-mathann} and \cite{LL}.

\begin{corollary}
\label{cor:bns-kahler}
Let $X$ be a compact K\"{a}hler manifold with $\pi_1(X)=G$. 
For any algebraically closed field $\k$ with trivial valuation 
and any representation $\sigma\colon G\to \GL_r(\k)$,
\[
\Sigma^1(G;\Z) \subseteq 
\rS\!\left(\Trop_\k\bigl(\VV^1(X,L_\sigma)\bigr)\right)^{\!\compl}.
\]
Equivalently, by Theorem~\ref{thm Delzant}, 
\[
\Trop_\k\bigl(\VV^1(X,L_\sigma)\bigr) \subseteq 
\bigcup_f \im\bigl(f^{*}\colon H^1(C_g;\R)\longrightarrow H^1(X;\R)\bigr),
\]
where the union runs over all equivalence classes of 
hyperbolic orbifold fibrations $f\colon X\to C_g$.
\end{corollary}

\begin{proof}
Since $\upsilon$ is trivial, $\Trop_\k(\VV^1(X,L_\sigma))$ is a cone over 
$\rS(\Trop_\k(\VV^1(X,L_\sigma)))$. The claim then follows from 
Theorems~\ref{thm:bns-twisted-trop} and~\ref{thm Delzant}.
\end{proof}

\subsection{Twisted Alexander polynomials of K\"{a}hler manifolds}
\label{subsec:twist-alex-kahler}

The next result generalizes \cite[Thm. 4.3(3)]{DPS-imrn}, 
giving a new obstruction for a finitely presented group being realizable 
as the fundamental group of a compact K\"{a}hler manifold. 

\begin{theorem}
\label{thm:alexpoly-kahler}
Let $X$ be a compact K\"ahler manifold with $\pi_1(X)=G$. 
Then for any representation $\sigma\colon G\to \GL_r(\k)$ over some field $\k$, 
the first twisted Alexander polynomial $\Delta^\sigma(X)$ of the pair $(X,\sigma)$ 
is either $0$ or $1$.
\end{theorem}

\begin{proof}
Without loss of generality, we can always assume that $\k$ is
algebraically closed. Next we assume that 
\[
\bigcup_f 
\im(f^{*}\colon H^1(C_g;\R)\longrightarrow H^1(X;\R)) \neq H^1(X;\R).
\]
By Hodge theory and the fact that the left side is a finite union, 
its codimension is at least $2$ in $H^1(X;\R)$. 
Therefore, $\codim \Trop_\k( \VV^1(X,L_\sigma)) \geq 2$ by Corollary \ref{cor:bns-kahler}. 
Since tropicalization preserves dimensions due to 
Theorem \ref{thm:structure-trop}, we have the variety $\VV^1(X,L_\sigma)$ 
has codimension at least $2$ in $\Hom(G,\k^\times)$. 
The zero locus of 
the twisted Alexander polynomial corresponds to the codimension one 
components of $\VV^1(X,L_\sigma)$. Hence, the first twisted Alexander 
polynomial $\Delta^\sigma(G)=1$ in this case.

Next, suppose that 
\[
\bigcup_f 
\im(f^{*}\colon H^1(C_g;\R)\longrightarrow  H^1(X;\R)) = H^1(X;\R).
\]
By Theorem~\ref{thm Delzant}, this happens if and only if 
there exists a hyperbolic orbifold fibration $f\colon X \to (C_g, \bm{\mu})$ 
such that $f^{*}\colon H^1(C_g;\R)\to H^1(X;\R)$ is an isomorphism. 
The map $f$ induces a short exact sequence of groups, 
\[
\begin{tikzcd}[column sep=18pt]
1\ar[r]& K \ar[r]&  G \ar[r, "f_{\#}"]&[2pt]  \Gamma
\ar[r]& 1,
\end{tikzcd}
\]
where $\Gamma=\pi_1^{\orb}(C_g, \bm{\mu})$ and $K=\ker(f_{\#})$ 
is a finitely generated group due to \cite[Lem.~3]{CKO}. 
Moreover, $f_{\#}$ induces an isomorphism between the maximal 
free abelian quotients of $G$ and $\Gamma$; in particular, 
$\Hom^0(G,\k^{\times}) \cong \Hom^0(\Gamma,\k^{\times})$.

Given a character $\rho \in \Hom^0(\Gamma,\k^{\times})$, 
let $f^* L_{\rho}$ denote the rank-one local system on $G$ 
obtained by pullback along $f_{\#}$. 
The Hochschild--Serre spectral sequence yields the exact sequence
\begin{equation}
\label{eq:hs-ses}
\begin{tikzcd}[column sep=16pt]
H_0(\Gamma; H_1(K;L_\sigma)\otimes L_\rho) \ar[r]
& H_1(G; L_\sigma\otimes f^*L_{\rho}) \ar[r]
& H_1(\Gamma; H_0(K;L_\sigma)\otimes L_\rho) \ar[r]& 0.
\end{tikzcd}
\end{equation}
Note that $(L_\sigma\otimes f^* L_{\rho})\vert_K = (L_\sigma)\vert_K$, 
so that $H_*(K; L_\sigma)$ is independent of $\rho$. We consider 
two sub-cases.

\smallskip
\noindent\textit{Sub-case (i): $H_0(K; L_\sigma)=0$.}
Then \eqref{eq:hs-ses} gives a surjection from
$H_0(\Gamma; H_1(K;L_\sigma)\otimes L_\rho)$
onto $H_1(G; L_\sigma\otimes f^*L_{\rho})$.

We claim that $H_1(K; L_\sigma)$ is a finite-dimensional $\k$-vector space.
Since $K$ is finitely generated by \cite[Lem.~3]{CKO}, say by $s$ elements 
$g_1,\ldots,g_s$, the partial free resolution 
$\k[K]^s \xrightarrow{d_1} \k[K] \to \k \to 0$ gives, upon tensoring with 
$V\coloneqq (L_\sigma)|_K \cong \k^r$, that $H_1(K;L_\sigma)$ is a quotient 
of $\ker(d_1\otimes V)\subset V^s \cong \k^{rs}$, hence finite-dimensional.

Consequently, $H_0(\Gamma; H_1(K;L_\sigma)\otimes L_\rho)\ne 0$ 
for at most finitely many $\rho \in \Hom^0(\Gamma, \k^\times)$.
Hence $\VV^1(X,L_\sigma)\cap \Hom^0(G,\k^{\times})$ is a finite set.
Since $b_1(G) = 2g \ge 2$, this finite set has no codimension-one 
components in $\Hom^0(G,\k^{\times})$, and so $\Delta^\sigma(G) = 1$.

\smallskip
\noindent\textit{Sub-case (ii): $H_0(K; L_\sigma)\ne 0$.}  
Let $L'$ denote the local system on $\Gamma$ with stalk 
$H_0(K; L_\sigma)$. 
By \eqref{eq:hs-ses} and the computations in Section~\ref{subsec:orbifolds}, 
applied to the orbifold group $\Gamma = \pi_1^{\orb}(C_g, \bm{\mu})$ 
with $g\ge 1$ and the local system $L'$, one of the 
following holds:
\begin{enumerate}[label=\textup{(\alph*)},itemsep=2pt]
\item $H_1(\Gamma; L'\otimes L_\rho) \ne 0$ for all 
$\rho \in \Hom^0(\Gamma, \k^\times)$, then by \eqref{eq:hs-ses} so is
$H_1(G; L_\sigma\otimes f^*L_{\rho})$, hence $\Delta^\sigma(G) = 0$ in this case; or
\item $H_1(\Gamma; L'\otimes L_\rho) = 0$ for all but 
finitely many $\rho$, in which case (combining with 
the finiteness of the left-hand term in \eqref{eq:hs-ses})
$\VV^1(X,L_\sigma)\cap \Hom^0(G,\k^{\times})$ is again 
finite, and so $\Delta^\sigma(G) = 1$.
\end{enumerate}

\smallskip\noindent
In all cases, $\Delta^\sigma(G)$ is either $0$ or $1$ 
(up to units), completing the proof.
\end{proof}

\begin{corollary}
\label{cor:Alexander-polynomial}
Let $X$ be a compact K\"ahler manifold with $\pi_1(X)=G$. 
Consider $\sigma=\id\colon G\to \GL_1(\k)$ the trivial representation 
over some field. If $b_1(X)\geq 4$, then
$\Delta^\sigma(X)=0$ if and only if  there exists a 
hyperbolic orbifold fibration $f\colon X \to C_g$ such that 
$f^*\colon H^1(C_g;\R)\to H^1(X;\R)$ is an isomorphism; 
otherwise $\Delta^\sigma(X)=1$.
In particular, it does not depend on $\ch(\k)$.
\end{corollary}

\begin{proof}
Following the proof of Theorem~\ref{thm:alexpoly-kahler}, 
we may assume that there exists a hyperbolic orbifold fibration 
$f\colon X \to (C_g, \bm{\mu})$ such that 
$f^*\colon H^1(C_g;\R)\to H^1(X;\R)$ is an isomorphism 
(otherwise $\Delta^\sigma(X)=1$ already, regardless of $\sigma$).

Since $\sigma=\id$ is the trivial rank-one representation, 
$(L_\sigma)|_K = \k_K$ is the constant sheaf on $K$, so 
$H_0(K;\, L_\sigma) = \k \ne 0$. Thus sub-case~(i) 
of the proof of Theorem~\ref{thm:alexpoly-kahler} does not occur, 
and we are in sub-case~(ii), with $L' = H_0(K;\k) = \k$ 
the trivial local system on $\Gamma = \pi_1^{\orb}(C_g,\bm{\mu})$.

It remains to decide between sub-cases (ii)(a) and (ii)(b). 
For the trivial local system $L' = \k$, 
$H_1(\Gamma;\, L' \otimes L_\rho) = H_1(\Gamma;\, L_\rho)$. 
Since $f^*$ is an isomorphism, 
$b_1(X) = b_1(C_g) = 2g$, so the condition $b_1(X) \geq 4$ 
forces $g \geq 2$. By formula~\eqref{eq:v1piorb}, 
$H_1(\Gamma;\, L_\rho) \ne 0$ for \emph{all} 
$\rho \in \Hom^0(\Gamma, \k^\times)$, placing us in sub-case~(ii)(a). 
Hence $\Delta^\sigma(G) = 0$, and this conclusion is independent 
of $\ch(\k)$ since formula~\eqref{eq:v1piorb} holds over any 
algebraically closed field.
\end{proof}

\subsection{Examples}
\label{subsec:examples}

In this subsection, we construct two examples from the weighted right angled Artin groups, 
which is a natural generalization of the right-angled Artin groups.

\begin{definition}
\label{def:weighted RAAG}
Let $\Gamma_\ell= (V, E, \ell)$ be an edge-weighted finite simple graph, with 
vertex set $V$, edge set $E$  and an edge weight function $\ell \colon E \to \Z_{>0}$.
The weighted right-angled Artin group associated to  $\Gamma_\ell$ is the group 
$G_{\Gamma_\ell}$ generated by the vertices $a\in V$, with a defining relation
 \[
 [a_i,a_j]^{\ell(e)}=1
 \]
for each edge $e = \{a_i,a_j\}$ in $E$ (here $[a_i,a_j]=a_i a_j a_i^{-1}a_j^{-1}$). 
If $\ell(e)=1$ for all $e\in E$, then $G_{\Gamma_\ell}$ is the classical right-angled Artin group.
\end{definition}

Such weighted right angled Artin groups are studied in \cite{LL}. In particular, 
the K\"ahler  weighted right angled Artin groups are classified in loc. cit..
\begin{remark}
\label{rem:WRAAG}
For an edge  with weight $m$ in $G_{\Gamma_\ell}$ connecting vertices 
$a$ and $b$, we have the relation $[a,b]^m=1$ in $G_{\Gamma_\ell}$. 
The abelianization of the Fox derivative of this relation gives 
\[
\begin{pmatrix}
  m(1-b)  & m(a-1) 
\end{pmatrix}^T,
\]
 which is  a column in the Alexander matrix. Such column vanishes if we 
 use field coefficient $\k$ and $\mathrm{char}(\k)$ divides $m$. 
 On the other hand, if we use complex coefficients, this column 
 is equivalent to
 \[
 \begin{pmatrix}
  1-b  & a-1 
\end{pmatrix}^T.
\]
In particular, if $\Gamma_\ell $ is a complete graph with $n$-vertices, 
we have that $\VV^1(G_{\Gamma_\ell},\C)=\VV^1(\Z^n,\C)=\{\bo\}$. 
\end{remark}

\begin{example}
\label{ex:weighted-artin}
Consider the group
\[
G= \langle a_1,a_2,a_3,a_4 \mid [a_1,a_i]=1 \text{ for } 2\leq i\leq 4,\;
[a_i,a_j]^2=1 \text{ for } 2\leq i<j\leq 4 \rangle.
\]
This is a weighted right-angled Artin group; it is shown in \cite{LL} that $G$
is not a K\"{a}hler group.

By Remark~\ref{rem:WRAAG}, the Fox-derivative columns contributed
by the weight-$2$ relations $[a_i,a_j]^2=1$ are equivalent over $\C$
to those of the corresponding unweighted commutator relations $[a_i,a_j]=1$.
Hence the Alexander matrix of $G$ over $\C$ coincides with that of
$\Z^4$, giving $\VV^1(G,\C)=\{\bo\}$. Since this variety has no
codimension-one component in $(\C^{\times})^4$, 
Proposition~\ref{prop:jump-loci-and-Alexander}\eqref{ja1} 
gives $\Delta^{\id}(G)=1$ over $\C$---consistent
with K\"{a}hlerianity, so the untwisted test over $\C$ yields no obstruction.

Over a field $\k$ with $\ch(\k)=2$, the weight-$2$ columns vanish
identically, and the Alexander matrix reduces to that of
$\Z \times F_3 = \langle a_1,a_2,a_3,a_4 \mid [a_1,a_i]=1,\, 2\le i\le 4\rangle$.
The unique codimension-one component of $\VV^1(G,\k)$ is then $\{t_1=1\}$,
so Proposition~\ref{prop:jump-loci-and-Alexander}\eqref{ja1} gives
$\Delta^{\id}(G)=t_1-1$ over $\k$. Since $t_1-1$ is neither $0$ nor $1$,
Theorem~\ref{thm:alexpoly-kahler} shows that $G$ is not a K\"{a}hler group.

This example illustrates that Theorem~\ref{thm:alexpoly-kahler} can
detect non-K\"{a}hlerianity via fields of positive characteristic even when
the untwisted test over $\C$ fails.
\end{example}

\begin{example}
\label{ex:KT-knot product with WRRAG}
Let $G$ be  the fundamental group of the complement of the Conway knot $K$.
Then the Alexander polynomial $\Delta_K(t) = 1$ and as shown by 
Friedl--Kim \cite[Sec.~5.2]{FK06}, $G$ admits a representation 
$\sigma_1\colon G\to \GL_4(\overline{\F}_{13})$ factoring 
through $S_5$, for which the twisted Alexander polynomial 
$\Delta_K^{\sigma_1}(t) \ne 1$. Here $\overline{\F}_{13}$ is 
an algebraically closed field with characteristic $13$. Hence we have 
$ \VV^1(G,\C)=\{\bo\}$ and $ \VV^1(G, \overline{\F}_{13}) $ consists of roots of 
$ \Delta_K^{\sigma_1}(t) $ and $1$.

For any $n\geq 2$, set $G_{2n-1}= \langle a_1,\dots,a_{2n-1} \mid  
[a_i,a_j]^{13}=1 \text{ for any }  1\leq i<j\leq 2n-1 \rangle.$
By Remark \ref{rem:WRAAG}, we have
$ \VV^1(G_{2n-1},\C)=\{\bo\}$ and $ \VV^1(G_{2n-1},\overline{\F}_{13})= 
(\overline{\F}_{13}^\times)^{2n-1}$. 
Using Proposition \ref{prop:product-formula}, we have 
$\VV^1( G\times G_{2n-1},\C)= \{\bo\}. $ By Proposition \ref{prop:jump-loci-and-Alexander}, 
the first Alexander polynomial over $\C$ of the group $G\times G_{2n-1}$ is $1$, which 
is consistent with K\"{a}hlerianity in \cite[Theorem 4.3(3)]{DPS-imrn}. So the classical Alexander 
polynomial gives no obstruction.
Note, moreover, that $b_1(G)=1$, while the abelianization of $G_{2n-1}$ 
is $\Z^{2n-1}$; hence $b_1(G\times G_{2n-1})=2n$. Since the first Betti 
number of a K\"{a}hler group must be even, the parity obstruction does 
not apply either. Thus neither of the two most immediate tests for 
K\"{a}hlerianity rules out the groups $G\times G_{2n-1}$.
    
On the other hand, consider the representation $\sigma'$ of $G\times G_{2n-1}$ given 
by the composed homomorphism 
\[
G\times G_{2n-1} \longrightarrow G \longrightarrow  \GL_4(\overline{\F}_{13}),
\]
where the first map is the quotient map and the second map is 
the representation $\sigma_1$. Then by  Proposition \ref{prop:product-formula}, 
$\VV^1( G\times G_{2n-1}, L_{\sigma'})$ is a finite union
of codimension-one tori of type $\{ t_1-\lambda\}$ with $\lambda$ 
either a root of $ \Delta_K^{\sigma_1}(t) $ or $1$.
Due to Proposition \ref{prop:jump-loci-and-Alexander},  
$G\times G_{2n-1}$  has non-trivial twisted Alexander polynomial 
over $\overline{\F}_{13}$, and thus it is not a K\"{a}hler group for any $n\geq 2$.
\end{example}

\section{$3$-manifold groups} 
\label{sect:3-manifold}

A $3$-manifold pair is a pair $(M,\phi)$, where $M$ is a compact, connected, 
orientable, $3$-manifold with toroidal or empty boundary, and $\phi \in 
\Hom(\pi_1(M),\Z)$ is  nontrivial. We say that a $3$-manifold pair $(M,\phi)$ 
fibers over $S^1$ if there exists a fibration $p \colon M \to S^1$ such that the 
induced map $p_*\colon \pi_1(M) \to \pi_1(S^1)\cong \Z$ coincides with $\phi$. 
We refer to such $\phi$ as a fibered class.

The \emph{Thurston norm}\/ $\| \phi \|_T$ of a class 
$\phi\in H^1(M;\Z)$ is defined as the infimum of $-\chi(S_0)$, 
where $S$ runs though all the properly embedded, oriented 
surfaces in $M$  dual to $\phi$, and $S_0$ denotes 
the result of discarding all components of $S$ which 
are disks or spheres. In \cite{Th}, Thurston  proved that 
$\|-\|_T$ defines a seminorm on $H^1(M;\Z)$, which can 
be extended to a continuous seminorm on $H^1(M;\R)$.

The unit norm ball, $B_T=\{\phi\in H^1(M;\R) \mid \|\phi \|_T\leq 1\}$,
is a rational polyhedron with finitely many sides and which is symmetric 
in the origin. Moreover, there are facets of $B_T$, called the {\em fibered faces}\/ 
(coming in antipodal pairs), so that a class $\phi \in H^1(M;\Z)$ fibers if and 
only if it lies in the cone over the interior of a fibered face.
In \cite[Thm.~E]{BNS}, Bieri, Neumann, and Strebel showed that 
the BNS invariant of $G=\pi_1(M)$ is the projection onto $\rS(G)$ of 
the open fibered faces of the Thurston norm ball $B_T$; in particular, 
$\Sigma^1(G)=-\Sigma^1(G)=\Sigma^1(G;\Z)$. 

On the other hand, the following results due to Friedl and Vidussi \cite{FV13}
shows that the non-fibered pair can be detected by the 
twisted Alexander polynomial.

\begin{theorem}[{\cite[Thm.~1.1]{FV13}}]
\label{thm:Friedl-Vidussi}
Let $(M,\phi)$ be a $3$-manifold pair. If $\phi$ is non-fibered, there exists
an epimorphism $\alpha \colon G=\pi_1(M)\surj Q$ onto a finite group such that, 
letting $\sigma\coloneqq\reg \circ \alpha \colon G \to \GL_{|Q|}(\Z)$, 
where $\reg$ denotes the regular representation of $Q$, we have
\[
\Delta^{\phi,\sigma}(M)=0.
\]
\end{theorem}

As a direct application, we show that the upper bound from 
Theorem \ref{thm:bns-twisted-trop} is sharp for a large class 
of $3$-manifold groups.

\begin{theorem}
\label{thm:sigma-3mfd} 
Let $M$ be a compact, connected, orientable, $3$-manifold with toroidal 
or empty boundary. Setting $G=\pi_1(M)$, we have
\begin{equation}
\label{eq:sigma-3mfd}
\Sigma^1(G) = \rS\Bigg(\overline{\bigcup_{\sigma} 
\Trop_\C\big(\VV^1(M,L_{\sigma}\otimes \C)\big)}\Bigg)^{\compl},
\end{equation}
where the union runs over all representations 
$\sigma\colon G\to \GL_r(\Z)$ with finite image.
\end{theorem}

\begin{proof}
Theorem~\ref{thm:bns-twisted-trop} and Corollary~\ref{cor:twist-bound}
give $\bigcup_{\sigma}\Trop_\C\big(\VV^1(M,L_{\sigma}\otimes \C)\big)
\subseteq \Sigma^1(G)^{\compl}$ for any collection of finite-image
representations $\sigma$.  Since $\Sigma^1(G)^{\compl}$ is closed,
taking closures preserves this containment, so
\[
\Sigma^1(G) \subseteq \rS\Bigg(\overline{\bigcup_{\sigma}
\Trop_\C\big(\VV^1(M,L_{\sigma}\otimes \C)\big)}\Bigg)^{\compl}.
\]

For the reverse inclusion, it suffices to show that every non-fibered 
primitive class $\phi \in H^1(M;\Z)$ lies in 
$\Trop_\C\big(\VV^1(M,L_{\sigma}\otimes\C)\big)$ for some 
finite-image representation $\sigma$.  Indeed, since $\Sigma^1(G)^{\compl}$
is a finite union of rational polyhedral cones \cite[Thm.~E]{BNS}, 
the primitive integral classes are dense in 
$\Sigma^1(G)^{\compl}\cap S(G)$, and any non-fibered rational class is a 
positive multiple of a primitive one.  Every point of $\Sigma^1(G)^{\compl}$
is therefore a limit of primitive integral non-fibered classes, 
hence belongs to the closure of 
$\bigcup_\sigma\Trop_\C\big(\VV^1(M,L_\sigma\otimes\C)\big)$
once those classes are shown to lie in the union itself.

Let $\phi\colon G\surj \Z$ be a non-fibered primitive class.
By Theorem~\ref{thm:Friedl-Vidussi}, there exists a representation 
$\sigma\colon G\to \GL_r(\Z)$ with finite image such that, upon 
extending $\sigma$ over $\C$, we have $\Delta^{\phi,\sigma}(M)=0$. 
The surjection $\phi$ induces embeddings 
\[
\phi^*\colon \Hom(\Z,\C^*) \longinj \Hom(G,\C^*)
\ \text{ and } \
\phi^\# \colon \Hom(\Z,\R) \longinj H^1(G;\R).
\]
By  Proposition \ref{prop:jump-loci-and-Alexander}\eqref{ja2}, 
$\Delta^{\phi,\sigma}(M)=0$ implies that 
$\C^*\cong \im(\phi^*) \subset \VV^1(M,L_\sigma\otimes \C)$, and therefore 
\[
\im(\phi^\#) = \R\cdot \phi \subset 
\Trop_\C\big(\VV^1(M,L_\sigma\otimes \C)\big). \qedhere
\]
\end{proof}

\begin{question}
\label{quest:face}
With the notation of Theorem~\ref{thm:sigma-3mfd}, let 
$F_1,\ldots,F_k$ be the non-fibered faces of the Thurston norm ball $B_T$, 
and let $C_j$ denote the closed cone over $F_j$.
For each non-fibered class $\phi \in \operatorname{int}(C_j)$, the 
Friedl--Vidussi theorem provides a finite-image representation 
$\sigma_\phi$ such that $\Delta^{\phi,\sigma_\phi}(M)=0$.
Can this choice be made uniform on each face, i.e., does there exist,
for each $j$, a single representation $\sigma_j$ such that
\[
C_j \subset \Trop_\C\big(\VV^1(M,L_{\sigma_j}\otimes \C)\big)?
\]
\end{question}

An affirmative answer would imply that the closure in Theorem~\ref{thm:sigma-3mfd} 
can be taken over at most $k$ representations, one for each non-fibered face 
(the union being finite, and each summand closed, no closure is then needed).

\begin{remark}
\label{rem:FV-nonfibered}
The Friedl--Vidussi argument does not yield a positive answer to 
Question~\ref{quest:face} directly.  For each non-fibered class 
$\phi \in \inter(C_j)$, their proof produces a finite 
quotient $\alpha_\phi \colon \pi_1(M) \surj G_\phi$ 
by applying a separability argument to a Thurston-norm minimizing 
surface dual to $\phi$.  Since that surface varies with $\phi$, 
the resulting representation $\sigma_{\alpha_\phi}$ varies as well, 
and their method gives no control over whether a single finite-image 
representation can certify $\Delta^{\phi,\sigma}(M)=0$ for all 
integral classes $\phi \in C_j$ simultaneously.
\end{remark}

%%%%%%%%%%%%%%%%%%%%%%%%%%%%%%

\newcommand{\arxiv}[1]
{\texttt{\href{http://arxiv.org/abs/#1}{arXiv:#1}}}
\newcommand{\arx}[1]
{\texttt{\href{http://arxiv.org/abs/#1}{arXiv:}}
\texttt{\href{http://arxiv.org/abs/#1}{#1}}}
\newcommand{\doi}[1]
{\texttt{\href{http://dx.doi.org/#1}{doi:#1}}}
\renewcommand{\MR}[1]
{\href{http://www.ams.org/mathscinet-getitem?mr=#1}{MR#1}}

\end{document}